\documentclass[11pt,twoside]{amsart}
\usepackage{latexsym,amssymb,amsmath}
\usepackage[all]{xy}
\usepackage{enumerate}
\usepackage{extpfeil}
\usepackage{afterpage}
\usepackage{float}
\usepackage{hyperref}
\usepackage{tikz}

\numberwithin{equation}{section}
\newtheorem{theorem}{Theorem}[section]
\newtheorem{lemma}[theorem]{Lemma}
\newtheorem{proposition}[theorem]{Proposition}
\newtheorem{corollary}[theorem]{Corollary}

\theoremstyle{definition}
\newtheorem{definition}[theorem]{Definition}
\newtheorem{procedure}[theorem]{Procedure}
\newtheorem{remark}[theorem]{Remark}
\newtheorem{example}[theorem]{Example}

\begin{document}

\title[The v-number of ideals of covers]
{The v-numbers and linear presentations of ideals of covers of graphs}

\thanks{The second and third authors were supported by SNII, M\'exico. 
The first author was supported by a scholarship from CONAHCYT, M\'exico.}

\author[H. Mu\~noz-George]{Humberto Mu\~noz-George}
\address{
Departamento de
Matem\'aticas\\
Cinvestav, Av. IPN 2508, 07360, CDMX, M\'exico.
}
\email{humbertomgeorge@gmail.com}

\author[E. Reyes]{Enrique Reyes}
\address{
Departamento de
Matem\'aticas\\
Cinvestav, Av. IPN 2508, 07360, CDMX, M\'exico.
}
\email{ereyes@math.cinvestav.mx}

\author[R. H. Villarreal]{Rafael H. Villarreal}
\address{
Departamento de
Matem\'aticas\\
Cinvestav, Av. IPN 2508, 07360, CDMX, M\'exico.
}
\email{rvillarreal@cinvestav.mx}
\thanks{Rafael H. Villarreal, Corresponding author, Departamento de
Matem\'aticas,
Cinvestav, Av. IPN 2508, 07360, CDMX, M\'exico, email address: rvillarreal@cinvestav.mx}

\begin{abstract} 
Let $G$ be a graph and let $J=I_c(G)$ be its ideal of covers. 
The aims of this work are to study the 
{\rm v}-number ${\rm v}(J)$ of $J$ and to study when $J$ is linearly 
presented using combinatorics and commutative algebra. 
We classify when ${\rm v}(J)$ attains 
its minimum and maximum  possible values in terms of the vertex covers of
the graph that satisfy the exchange property. 
If the cover ideal of a graph has a linear
presentation, we express its v-number in terms of the covering number 
of the graph. If $G$ is unmixed, the graph $\mathcal{G}_J$ of $J$ is
the graph whose vertices are the minimal vertex covers of $G$ and
whose edges are the pairs $\{C,C'\}$ such that $|C\cup C'|=|C|+1$. 
We show necessary and sufficient conditions for the graph
$\mathcal{G}_J$ of $J$ to be connected. Then, for unmixed K\"onig
graphs, we classify when $J$ is linearly presented using graph theory,
and show some results on
Cohen--Macaulay K\"onig graphs. If $G$ is unmixed, it is shown that 
the columns of the linear syzygy matrix of $J$ are linearly
independent if and only if $\mathcal{G}_J$  has no strong 
$3$-cycles. One of our main theorems shows
that if $G$ is unmixed and has no induced $4$-cycles, then $J$ is linearly presented. 
For unmixed graphs without $3$- and $5$-cycles, we classify
combinatorially when $J$ is linearly presented.
\end{abstract}

\maketitle

\section{Introduction}\label{intro-section}

Let $G$ be a simple graph with vertex 
set $V(G)=\{t_1,\ldots,t_s\}$ and edge set $E(G)$.  
A subset $C\subset V(G)$ is a \textit{vertex cover} of $G$  
if every edge of $G$ is incident with at least one vertex in $C$.
A \textit{minimal vertex cover} of $G$ is a vertex cover which does not contain any vertex cover
properly. The dual
concept of a vertex cover is a \textit{stable} set, i.e., 
a subset of $V(G)$ is a stable vertex
set if and only if its complement with respect to $V(G)$ is a vertex
cover. For convenience 
of notation we write edges in set notation $\{t_i,t_j\}$.

Let $S=K[t_1,\ldots,t_s]=\bigoplus_{d=0}^{\infty} S_d$ be a polynomial ring over
a field $K$ with the standard grading. The monomials of $S$ are denoted 
$t^a:=t_1^{a_1}\cdots t_s^{a_s}$, $a=(a_1,\dots,a_s)\in\mathbb{N}^s$.
Let $C_1,\ldots,C_r$ be the minimal vertex covers
of $G$ and let $t_{c_1},\ldots,t_{c_r}$ be their corresponding
monomials, that is, $t_{c_j}=\prod_{t_i\in C_j}t_i$ for
$j=1,\ldots,r$. 
The following squarefree monomial ideals are associated to $G$.

The \textit{edge ideal} of $G$, denoted $I(G)$, is the
ideal \cite{cmg}:
$$
I(G):=(\{t_it_j\mid \{t_i,t_j\}\in E(G)\})\subset S,
$$
and the \textit{ideal of
covers} of $G$, denoted $I_c(G)$, is the ideal \cite{alexdual}:  
$$
I_c(G):=(t_{c_1},\ldots,t_{c_r})\subset S.
$$
\quad The aims of this work are to study the 
{\rm v}-number  of $I_c(G)$ and to 
study when $I_c(G)$ is linearly presented using combinatorics.  
In Section~\ref{prelim-section} we
present some well-known results that will be needed throughout the paper
and introduce some notation. 

We say a vertex cover $C$ of $G$ has the \textit{exchange 
property} if there exists $\{t_k,t_\ell\}\in E(G)$
such that $t_k\in C$, $t_\ell\notin C$, and
$(C\setminus\{t_k\})\cup\{t_\ell\}$ is a vertex cover of $G$. Vertex
covers with the exchange property always exist, and can be classified 
using neighbor sets and symmetric differences of vertex covers
(Lemma~\ref{jul13-24}, 
Proposition~\ref{equiv-exchange}). The \textit{exchange number} of
$G$, denoted $ \alpha_e(G)$, is the  
number of vertices in any smallest vertex cover of $G$ with the
exchange property. This number is related to the \textit{covering
number} $\alpha_0(G)$ of $G$ which is the number of
vertices in any smallest vertex cover of $G$. The v-{\em number} of
$I_c(G)$, denoted ${\rm 
v}(I_c(G))$, is an invariant that was introduced 
in \cite{min-dis-generalized} to study the asymptotic behavior 
of the minimum distance of projective Reed--Muller-type codes
\cite[Corollary~4.7]{min-dis-generalized}. This invariant 
is defined for any graded ideal, and 
can be computed using the software system \textit{Macaulay}$2$
\cite{mac2} and the algorithm given in 
\cite[Appendix A]{im-vnumber}. The {\rm v-number} of monomial ideals
has been studied in 
\cite{Biswas-Mandal,civan,Ficarra,v-number,Kumar-Nanduri-Saha,Saha,Saha-Gorenstein,saha-sengupta}.
Using a combinatorial
description for the v-number of ideals of clutters
\cite[Theorem~3.5]{v-number}, \cite[Remark~2.5]{Saha}, we obtain: 
$${\rm v}(I_c(G))=\alpha_e(G)-1.$$
\quad In general, ${\rm v}(I_c(G))\geq\alpha_0(G)-1$
\cite[Lemma~3.16]{v-number} with equality 
if and only if there exists a vertex cover $C$ of $G$ with 
$\alpha_0(G)$ vertices satisfying the exchange property
(Lemma~\ref{v-number-dual}).

Letting $J:=I_c(G)$, we say $J$ is \textit{linearly presented} if the monomials 
$t_{c_1},\ldots,t_{c_r}$ are of the same degree $\alpha_0(G)$ and 
the kernel of the presentation map $\psi\colon S^r\rightarrow S$, 
$e_i'\mapsto t_{c_i}$, is generated by vectors whose entries 
are linear forms, where $e_i'$ is the $i$-th unit vector in $S^r$. 

Let $G$ be an \textit{unmixed} graph, that is, $|C_i|=\alpha_0(G)$
for all $i$. Unmixed graphs are also 
called \textit{well-covered} graphs \cite{Kimura-Terai}. 
The graph of $J$, denoted by $\mathcal{G}_{J}$, has 
vertex set $V(\mathcal{G}_{J}):=\{C_i\}_{i=1}^r$, 
and the pair $\{C_i,C_j\}$ is an edge of $\mathcal{G}_{J}$ if and
only if 
$|C_i\cup C_j|=|C_i|+1$. 
According to \cite[Proposition~2.1]{Bigdeli-etal},
\cite{Dao-Eisenbud}, the ideal $J$ is linearly
presented if and only if for all $C_i, C_j\in V(\mathcal{G}_J)$,
$i\neq j$,  
there is a path in $\mathcal{G}_J$ 
connecting $C_i$ and $C_j$ whose vertices are 
contained in $C_i\cup C_j$ (cf.~Proposition~\ref{Bigdeli-Herzog-j}). 
If $J$ is linearly presented, then $\mathcal{G}_J$ is connected but the converse
is not true  
(Example ~\ref{example-conjetura2}).

Our first result relates the edges of the graph $\mathcal{G}_J$ of $J$ with those of 
the graph $G$, classifies when ${\rm v}(J)$ attains its minimum
possible value in terms of $\mathcal{G}_J$, and
determine the v-number of $J$ for a wide family of graphs. 
Recall that a graph without edges is called a \textit{discrete graph}.   

\noindent\textbf{Theorem~\ref{eclipse-apr8-24}.}\textit{ Let $G$ be
an unmixed graph and let $J=I_c(G)$. The following hold.  
\begin{enumerate}
\item[\rm(a)] $\{C_i,C_j\}$ is an edge of $\mathcal{G}_J$ if and only
if there exists $\{t_k,t_\ell\}\in E(G)$ such that $t_k\in C_i$,
$t_\ell\notin C_i$ and $(C_i\setminus\{t_k\})\cup\{t_\ell\}=C_j$. 
\item[\rm(b)] $\mathcal{G}_J$ is not a discrete graph if and only if 
${\rm v}(I_c(G))=\alpha_0(G)-1$.
\item[\rm(c)] If $J$ is linearly presented, then all 
minimal vertex covers of $G$ have the exchange property and 
${\rm v}(J)=\alpha_0(G)-1$.
\end{enumerate}
}

Given a vertex  $t_i$ of $G$, we denote the neighborhood of
$t_i$ by $N_G(t_i)$ and the closed neighborhood of $t_i$ by 
$N_G[t_i]$. 
If $N_G(t_k)\subset N_G[t_\ell]$ for some
$\{t_k,t_\ell\}\in E(G)$, then there exists a minimal vertex cover $C$
of $G$ such that $N_G(t_k)\setminus
C=\{t_\ell\}$ (Proposition~\ref{saha->exchange}). 
The converse is not true (Example~\ref{5-cycle}). 
The condition ``$N_G(t_k)\subset N_G[t_\ell]$ for some
$\{t_k,t_\ell\}\in E(G)$'' first appeared 
in \cite[Proposition~3.1]{Saha}. This condition implies the 
inequality ${\rm v}(I_c(G))\leq{\rm bight}(I(G))-1$ \cite{Saha}, where
${\rm bight}(I(G))$ is the \textit{big height} of $I(G)$, that is, 
the size of a largest minimal vertex cover of
$G$. 

As usual, the regularity and Krull dimension of the quotient ring $S/I_c(G)$ are
denoted ${\rm reg}(S/I_c(G))$ and $\dim(S/I_c(G))$, respectively 
(Section~\ref{prelim-section},
Definition~\ref{regularity-socle-degree}). By 
\cite[Theorem~3.31]{edge-ideals} and \cite[Proposition~3.2]{v-number}, 
one has the inequalities
$$
{\rm bight}(I(G))-1\leq{\rm reg}(S/I_c(G))\leq\dim(S/I_c(G)),
$$ 
with equality on the left if $I(G)$ is sequentially Cohen--Macaulay (cf.
Corollary~\ref{apr10-24}).  
By a result of Saha
\cite[Theorem~3.8]{Saha} the v-number of $I_c(G)$ is a lower bound 
for the regularity of $S/I_c(G)$. Thus, in general, one has the
following inequalities
\begin{equation}\label{summarize-ineq}
\alpha_0(G)-1\leq {\rm v}(I_c(G))\leq{\rm
reg}(S/I_c(G))\leq\dim(S/I_c(G))=s-2.
\end{equation}

Let $G$ be a graph and let $\overline{G}$ be its complement. We say
$G$ is a \textit{codi-graph} (complement disconnected graph) 
if $\overline{G}$ is disconnected. A graph $G$ in which every
connected induced subgraph has a disconnected complement is called a
\textit{cograph}. A connected cograph is a codi-graph but the converse
is not true (Example~\ref{example-cograph}). 
Let $G_1,\ldots,G_k$ be mutually (vertex) disjoint graphs. Their 
\textit{join}, denoted $G_1\ast\cdots*G_k$, consist
 of $G_1\cup\cdots\cup G_k$ and all the edges of the form $\{x,y\}$, 
 where $x\in V(G_i)$ and $y\in V(G_j)$ for all $i\neq j$. 
If $k\geq 2$ and $G_i$ is a discrete graph for all
 $i$, 
we say that $G$ is a \textit{complete 
multipartite graph} or a \textit{complete} $k$-\textit{partite
graph}, 
and if $|V(G_i)|=p$ for all $i$ we say that $G$ is
\textit{homogeneous}.  

We classify when the v-number of $I_c(G)$ attains its maximum possible
value in Eq.~\eqref{summarize-ineq} and give a combinatorial
description of complete multipartite graphs.

\noindent \textbf{Theorem~\ref{v-ci-maximum=s-2}.}\textit{
Let $G$ be a graph with $E(G)\neq\emptyset$. 
The following conditions are equivalent.
\begin{enumerate}
\item[\rm(a)] ${\rm v}(I_c(G))=s-2$.
\item[\rm(b)] ${\rm v}(I_c(G))={\rm
reg}(S/I_c(G))=\dim(S/I_c(G))=s-2$ and $G$ is a codi-graph.
\item[\rm(c)] Every vertex cover of $G$ with
the exchange property has exactly $s-1$ vertices.
\item[\rm(d)] ${\rm v}(I_c(H))=|V(H)|-2$ for any induced subgraph $H$ of
$G$ with $E(H)\neq\emptyset$.
\item[\rm(e)] Any induced subgraph $H$ of
$G$ with $E(H)\neq\emptyset$ is connected.
\item[\rm(f)] $G$ is a complete multipartite graph.
\end{enumerate}
} 

We recover the fact that ${\rm v}(I_c(G))=s-2$ 
for complete multipartite
graphs \cite[Theorem~3.6]{Saha}, and show 
that $I_c(G)$ is Cohen--Macaulay when ${\rm v}(I_c(G))=s-2$
(Corollary~\ref{vp=s-2}). If $G$ is the join of two graphs $G_1$ and $G_2$, we give a formula 
to compute the v-number of $G$ in terms of those of $G_1$ and $G_2$ 
(Proposition~\ref{join-formula}, Example~\ref{example-join-formula}).
Given an integer $k\geq 1$, let $G_1$ be the vertex disjoint union of $k+1$
edges, let $G_2$ be a discrete
graph with one vertex, and let $G=G_1*G_2$ be their join. Then
(Corollary~\ref{diff}, cf.~\cite[Theorem~3.10]{Saha}):
$$
{\rm reg}(S/I_c(G))-{\rm v}(I_c(G))=k.
$$

An edge $e=\{t_1,t_2\}$ of $G$ has the  {\rm (P)} \textit{property}, if for all edges 
$\{t_1,t'_1\}$ and $\{t_2,t'_2\}$ distinct from $e$, one has that 
$\{t'_1,t'_2\}$ is an edge of $G$. 
We characterize when all (resp. none) of the edges of a
complete 
multipartite graph have the {\rm (P)} property. For these graphs 
we characterize the unmixed property
(Proposition~\ref{compl-$k$-partite}). If $G$ is a homogeneous complete multipartite 
graph, it is shown that $\mathcal{G}_{I_c(G)}$ is connected 
if and only if ${\rm v}(I_c(G))=\alpha_0(G)-1$
(Corollary~\ref{jul14-24-3}). 

Let $G$ be an unmixed graph and let $r$ be the number of minimal
vertex covers of $G$. The 
\textit{linear syzygy matrix} of $J=I_c(G)$, denoted
$\mathcal{LS}(J)$, is an $r\times|E(\mathcal{G}_J)|$ matrix whose entries are
either $0$ or have the form $\pm t_i$, $t_i\in S$, and whose columns correspond
to linear syzygies of $J$ (Section~\ref{prelim-section}). A $k\times
k$ \textit{minor} of a matrix $\varphi$ with entries in $S$ is the determinant of a
square submatrix of $\varphi$ of order $k$. 
The ideal generated by the $k\times k$ minors of $\varphi$ is denoted by $I_k(\varphi)$, the
(\textit{determinantal}) \textit{rank} of $\varphi$ and
\textit{height} of $I_k(\varphi)$ are denoted ${\rm rank}(\varphi)$
and ${\rm ht}(I_k(\varphi))$, respectively \cite{Ramos-Simis}.

If $G$ is an unmixed graph, then 
$J$ is linearly presented if and only ${\rm
rank}(\mathcal{L}\mathcal{S}(J))=r-1$ 
and ${\rm ht}(I_{r-1}(\mathcal{L}\mathcal{S}(J))\geq 2$
(Proposition~\ref{syz}). Any nonzero minor of $\mathcal{LS}(J)$ is of
the form $\pm t^a$, $t^a\in S$  
(Lemma~\ref{chicapitanga}, cf.~\cite[cf.~Lemma~3.5]{birational}).
For unmixed graphs, we show that  $\mathcal{G}_{J}$ is connected 
if and only if ${\rm rank}(\mathcal{LS}(J))=r-1$
(Proposition~\ref{apr14-05-coro}). 

Let $G$ be an unmixed graph. We describe the induced paths of 
length $2$ of $\mathcal{G}_J$ 
(Proposition~\ref{jun5-24}) and the paths $\{C_1,\ldots,C_n\}$ 
of $\mathcal{G}_J$ such that $C_i\subset C_1\cup C_n$ for all $i$
(Proposition~\ref{jul14-24}). An edge $e=\{t,t'\}$ of $G$ is 
called an \textit{exchange edge}, if there are
minimal vertex covers $C$ and $C'$ of $G$ such that
$C'=(C \setminus \{t\}) \cup \{t'\}$.
If $G$ has no isolated vertices and $\mathcal{G}_J$ is connected, then $V(G)$ is covered by the
exchange edges (Proposition~\ref{connected-exchange}). 
If $\mathcal{G}_J$ is a tree or a path, we classify when $J$ has a linear
presentation (Corollary~\ref{linearp-trees}) and in the latter case 
we give a combinatorial criterion (Proposition~\ref{jul14-24-1}). In
these cases the reader is refer to \cite[Sections~3 and 4]{MSJ} for
additional results on the structure of $J$. 
If $G$ is an unmixed bipartite graph without isolated
vertices, we show that $\mathcal{G}_J$ is bipartite
(Proposition~\ref{I_c-bipartite}).

A $3$-cycle $\{C_1, C_2, C_3, C_1\}$ of $\mathcal{G}_{J}$
is \textit{strong} if $C_2\setminus C_1=C_3\setminus C_1$. For unmixed
graphs, we show that the set $F$ of column vectors of
$\mathcal{L}\mathcal{S}(J)$ is a minimal system of generators
for $S(F)$, the $S$-module generated by $F$, if and only if $\mathcal{G}_J$  has no strong 
$3$-cycles (Proposition~\ref{min-gen-S}).

Two distinct vertices $t$ and $t'$
of a graph $G$ are called \textit{duplicated} if
$N_G(t)=N_G(t')$. If $\mathcal{G}_{J}$ is connected, 
we show that $G$ has no duplicated vertices and that 
every  $4$-cycle of $G$ has an edge that does not
satisfy the {\rm (P)} property (Theorem~\ref{connected-duplicated}). 
The converse is not true (Example~\ref{example-conjetura}). 

Let $P$ be a \textit{matching} of a graph $G$, that is, $P$ is a set
of disjoint edges of $G$.
If $P$ is a partition of $V(G)$, we call $P$ a \textit{perfect
matching}.  We say that $P$ has the {\rm (P)} \textit{property}   
if each edge of $P$ has the {\rm (P)} property. A graph $G$ is 
\textit{K\"onig} if 
$G$ has a matching with $\alpha_0(G)$ edges. K\"onig graphs 
include bipartite graphs by a classical result of K\"onig
(Theorem~\ref{konig-theorem}), as well as whisker graphs and very
well-covered graphs \cite[p.~278]{unmixed-c-m}. A perfect matching $P$ of $G$ is of \textit{K\"onig
type} if $\vert P \vert=\alpha_0(G)$ \cite{MRV}. This notion relates to the 
fact that very well-covered graphs (Definition~\ref{vwc}) can be classified using 
perfect matchings \cite{favaron,Kimura-Terai} (cf. Proposition~\ref{vw-c}).

A graph $G$ is called \textit{Cohen--Macaulay} if $I(G)$ is
Cohen--Macaulay (Definition~\ref{regularity-socle-degree}). 
We show the following two results on Cohen--Macaulay 
K\"onig  graphs. 

\noindent \textbf{Theorem~\ref{Konig-connectd}.}\textit{ 
 Let $G$ be an unmixed K\"onig graph without isolated vertices. The following
conditions are equivalent. 
\begin{enumerate}
\item[\rm(a)] $G$ is Cohen--Macaulay.
\item[\rm(b)] $I_c(G)$ is linearly presented.
\item[\rm(c)] $\mathcal{G}_{I_c(G)}$ is connected.
\item[\rm(d)] $G$ has no duplicated vertices.
\item[\rm(e)] Every induced $4$-cycle has an edge without the {\rm (P)} property.
\item[\rm(f)] $G$ has a unique perfect matching.
\end{enumerate}
}

If $G$ is a very well-covered graph, by a result of Crupi, Rinaldo and
Terai \cite[Theorem~0.2]{Crupi-Rinaldo-Terai}, 
$G$ is Cohen--Macaulay if and only
if $G$ has a unique perfect matching. Part (a) of Theorem~\ref{CM-connectd1}
gives a combinatorial description of the unique perfect matching of a 
Cohen--Macaulay K\"onig graph without isolated vertices. For a
bipartite graph $G$, Zaare-Nahandi \cite{zaare-nahandi} revisited the Cohen--Macaulayness of
$G$ and proved, by different methods than those of \cite{Crupi-Rinaldo-Terai},  
that $G$ is Cohen--Macaulay if and only if $G$ is unmixed and 
has a unique perfect matching.

\noindent \textbf{Theorem~\ref{CM-connectd1}.}\textit{
If $G$ is a Cohen--Macaulay K\"onig graph without isolated vertices, then 
\begin{enumerate}
\item[\rm(a)] $G$ has a unique perfect matching $P$ which is the set
of exchange edges of $G$, 
and
\item[\rm(b)] $\mathcal{G}_{I_c(G)}$ is a bipartite graph.
\end{enumerate}
}

Let $C_1\triangle C_2$ denote the symmetric difference of $C_1$ and
$C_2$. If $G$ is unmixed and $I_c(G)$
is not linearly presented, we show that $G$ has 
distinct minimal vertex covers $C_1$ and $C_2$ 
such that the induced subgraph $G[C_1 \triangle C_2]$ of $G$, with
vertex set $C_1 \triangle C_2$, is an unmixed bipartite graph without
isolated vertices which is not Cohen--Macaulay (Theorem~\ref{noLinear-noCM}).  
In general the converse is not true
(Example~\ref{example-conjetura3}). 

We come to one of our main results.

\noindent \textbf{Theorem~\ref{4-cycles-connected}.}\textit{
If $G$ is an unmixed graph without induced $4$-cycles, then
$I_c(G)$ is linearly presented and ${\rm v}(I_c(G))=\alpha_0(G)-1$.
}

In general the converse of
Theorem~\ref{4-cycles-connected} is not 
true (Example ~\ref{contraexample}). 

The set of all minimal vertex covers of a graph $G$ is denoted by
$E(G^\vee)$. The reason for this notation is explained in
Section~\ref{prelim-section} using the notion of
\textit{blocker} coming from the theory
of clutters. 
Let $G$ be an unmixed graph such that $I_c(G)$ is linearly presented.
If  $G[C_1 \triangle C_2]$ is unmixed and
not Cohen-Macaulay, where $C_1$ and $C_2$
are distinct minimal vertex covers of $G$, we show that there exists 
$C\in E(G^\vee)$ such that $C_1 \cap C_2  \not\subset C \subset C_1
\cup C_2$ (Theorem~\ref{not-linear-presented}).

We come to another of our main results. 

\noindent \textbf{Theorem~\ref{3-5-cycles-linpresented}.}\textit{
 Let $G$ be an unmixed graph without $3$- and $5$-cycles. The following condition are
 equivalent.
\begin{enumerate}
\item[\rm(a)] $I_c(G)$ is linearly presented.
\item[\rm(b)] If $G[C_1 \triangle C_2]$ is unmixed, where 
$C_1, C_2 \in E(G^\vee)$, $C_1\neq C_2$, then it is Cohen--Macaulay.
\end{enumerate}
}

Consider the following families of graphs: 
\begin{enumerate}
\item[\rm(a)] $U$ is the family of unmixed graphs. 
\item[\rm(b)] $U_1=\{G \in U \mid \mbox{every induced } 4 \mbox{-cycle of }  G\mbox{ has an edge without the {\rm (P)} property}\} $.
\item[\rm(c)] $U_2=\{G \in U \mid  G \mbox{ has no duplicated vertices}\} $.
\item[\rm(d)] $U_3=\{G \in U \mid  \mathcal{G}_{I_c(G)} \mbox{ is connected}\} $.
\item[\rm(e)] $U_4=\{G \in U \mid  I_c(G) \mbox{ is linearly presented}\} $.
\item[\rm(f)] $U_5=\{G \in U \mid  G \mbox{ has no induced} \mbox{ }4 \mbox{-cycles}\}$.
\end{enumerate}
\quad We show the inclusions 
$U_5 \subset U_4 \subset U_3 \subset U_2 \subset U_1 \subset U,
$
and give examples to prove that all inclusions are strict 
(Proposition~\ref{jul27-24}).

In Section~\ref{examples-section}, we present examples related to 
some of our results. Then, in Appendix~\ref{Appendix}, we give the
procedures for \textit{Macaulay}$2$ \cite{mac2} that we use 
to compute the graph of ideals of covers, and regularities, v-numbers, 
linear syzygy matrices, and ideals of minors of monomial ideals.   

For all unexplained
terminology and additional information  we refer to
\cite{Har,edge-ideals,van-tuyl-survey,monalg-rev} (for graph theory and edge ideals), and
\cite{Eisen,Mats,Ramos-Simis} (for commutative rings and ideals of minors).

\section{Preliminaries}\label{prelim-section}

To avoid repetitions, we continue to employ
the notations and
definitions used in Section~\ref{intro-section}. 
Let $S=K[t_1,\ldots,t_s]=\bigoplus_{d=0}^{\infty} S_d$ be a polynomial ring over
a field $K$ with the standard grading  
and let $\mathcal{C}$ be a
\textit{clutter} with vertex 
set $V(\mathcal{C})=\{t_1,\ldots,t_s\}$, that is, $\mathcal{C}$ is a family of subsets of
$V(\mathcal{C})$, called \textit{edges}, none of which is included in
another. Any squarefree monomial ideal is
the edge ideal of a clutter 
\cite[pp.~220--221]{monalg-rev}. 
An \textit{isolated vertex} of $ \mathcal{C}$ is a vertex
that is not in any edge of $\mathcal{C}$. 
The set of edges of $\mathcal C$ is denoted by
$E(\mathcal{C})$. The primer example of a clutter is a simple graph
$G$. The \textit{edge ideal} of $\mathcal{C}$, denoted
$I(\mathcal{C})$, 
is the ideal of $S$ generated
by all squarefree monomials $t_e:=\prod_{t_i\in e}t_i$ such that
$e\in E(\mathcal{C})$.  
In what follows $I$ denotes the edge ideal of
$\mathcal{C}$. 

A prime ideal $\mathfrak{p}$ of $S$ is an \textit{associated prime}
of $I$ if
$(I\colon f)=\mathfrak{p}$ for some $f\in S_d$, where $(I\colon f)$ is
the set of all $g\in S$ such that $gf\in I$. The set of associated primes of $I$ 
is denoted by ${\rm Ass}(I)$. The v-{\em number} of $I$, denoted ${\rm
v}(I)$, is the following invariant of
$I$ \cite[Corollary~4.7]{min-dis-generalized}:
$$
{\rm v}(I):=\min\{d\geq 0 \mid\exists\, f 
\in S_d \mbox{ and }\mathfrak{p} \in {\rm Ass}(I) \mbox{ with } (I\colon f)
=\mathfrak{p}\}.
$$
\quad 
One can define the v-number of $I$ locally at each associated 
prime $\mathfrak{p}$ of $I$\/:
$$
{\rm v}_{\mathfrak{p}}(I):=\mbox{min}\{d\geq 0\mid \exists\, f\in S_d
\mbox{ with }(I\colon f)=\mathfrak{p}\}.
$$
\quad A subset $A$ of $V(\mathcal{C})$ is called 
{\it independent\/} or {\it
stable\/} if $e\not\subset A$ for any  
$e\in E(\mathcal{C})$. If $\mathcal{C}$ is a
graph, this means that for $x,y\in A$, $\{x,y\}\notin
E(\mathcal{C})$. The dual concept of a stable set
is a {\it vertex cover\/}, i.e., a subset $C$ of $V(\mathcal{C})$ is a vertex
cover if and only if $V(\mathcal{C})\setminus C$ is a stable set. A 
{\it minimal vertex cover\/} is a vertex cover which
is minimal with respect to inclusion.  
If $A$ is a stable set of $\mathcal{C}$, the \textit{neighbor set} of $A$, denoted 
by $N_\mathcal{C}(A)$, is the set of all vertices $t_i$ 
such that $\{t_i\}\cup A$ contains an edge of $\mathcal{C}$. 
For use below, $\mathcal{A}_\mathcal{C}$
denotes the family of all stable sets $A$ of $\mathcal{C}$ whose neighbor
set $N_\mathcal{C}(A)$ is a minimal vertex
cover of $\mathcal{C}$. The following result shows that the 
v-number of a clutter is a combinatorial invariant.

\begin{theorem}\cite[pp.~9--10]{v-number}\label{v-number-clutters-graphs}
Let $I$ be the edge ideal of a clutter $\mathcal{C}$. If $I$ is not
prime, then, 
$$
\mathrm{v}(I)=\min\{|A|\colon 
A\in\mathcal{A}_\mathcal{C}\}\ \mbox{ and }\ (I\colon
t_A)=(N_\mathcal{C}(A))\ \forall\, A\in\mathcal{A}_\mathcal{C}.
$$ 
\end{theorem}

The number of
vertices in any smallest vertex cover of $\mathcal{C}$, denoted 
$\alpha_0({\mathcal C})$, is called the \textit{covering
number} of $\mathcal{C}$. The \textit{height} ${\rm ht}(I)$ of $I$ is
$\alpha_0(\mathcal{C})$ and the \textit{Krull dimension} $\dim(S/I)$ of $S/I$ 
is $s-{\rm ht}(I)$.   
A clutter is
\textit{well-covered} if every maximal
stable set is a maximum stable set. 
The clutter of minimal vertex covers of $\mathcal{C}$, denoted 
$\mathcal{C}^\vee$, is called the {\it blocker} of $\mathcal{C}$. 
The edge set $E(\mathcal{C}^\vee)$ of the blocker 
of $\mathcal{C}$ is the set of minimal vertex covers of $\mathcal{C}$
and the vertices of $\mathcal{C}^\vee$ are the vertices of $\mathcal{C}$. The 
edge ideal of $\mathcal{C}^\vee$, denoted by $I_c(\mathcal{C})$, is
called the {\it ideal of covers\/} of 
$\mathcal{C}$.  

The \textit{deletion} of a vertex $t_i$, denoted by
$\mathcal{C}\setminus t_i$,  is the clutter obtained from
$\mathcal{C}$ by deleting the vertex $t_i$ and all edges containing
$t_i$. Let $A$ be a set of vertices of $\mathcal{C}$.  The \textit{deletion} of $A$, denoted
$\mathcal{C}\setminus A$, is obtained by successively deleting
the vertices of $A$. The \textit{induced clutter} on $A$, 
denoted $\mathcal{C}[A]$, is the clutter on $A$ with 
edge set consisting of all edges of $\mathcal{C}$ contained in $A$. 

A graded ideal is called \textit{unmixed} if all its associated
primes have the same height. A clutter is called \textit{unmixed} if
all its minimal vertex covers have the same number of elements. 
The next result relates minimal vertex covers and
associated primes.    

\begin{lemma}\cite[Lemma~6.3.37]{monalg-rev}\label{jul1-01} Let $C$
be a set of vertices of a clutter
$\mathcal{C}$. Then, $C$ is a minimal vertex cover of $\mathcal{C}$
if and only if the ideal of $S$ generated by $C$ 
is an associated prime of $I(\mathcal{C})$.
\end{lemma}

\begin{definition}\label{regularity-socle-degree}\rm Let $I\subset S$ be a graded ideal and let
${\mathbf F}$ be the minimal graded free resolution of $S/I$ as an
$S$-module:
\[
{\mathbf F}:\ \ \ 0\longrightarrow
\bigoplus_{j}S(-j)^{b_{g,j}}
\stackrel{\varphi_g}{\longrightarrow} \cdots
\longrightarrow\bigoplus_{j}
S(-j)^{b_{1,j}}\stackrel{\varphi_1}{\longrightarrow} S
\longrightarrow S/I \longrightarrow 0.
\]
\quad The ideal $I$ has a $d$-\textit{linear
resolution} if
all maps $\varphi_i$, $i\geq 2$, are defined by matrices whose
entries are linear forms and all entries of $\varphi_1$ are forms of degree
$d$. The \textit{projective dimension} of $S/I$, denoted ${\rm
pd}_S(S/I)$, is equal to $g$. 
The {\it Castelnuovo--Mumford regularity\/} of $S/I$ ({\it
regularity} of $S/I$ for short) is defined as
$${\rm reg}(S/I)=\max\{j-i \mid b_{i,j}\neq 0\}.
$$
\quad If ${\rm
pd}_S(S/I)$ is the height of $I$, we say the ring $S/I$ and the ideal 
$I$ are \textit{Cohen--Macaulay}. 
\end{definition}
An element $f\in S$ is called a {\it zero-divisor\/} of $S/I$---as an
$S$-module---if there is
$\overline{0}\neq \overline{a}\in S/I$ such that
$f\overline{a}=\overline{0}$, and $f$ is called {\it regular\/} on
$S/I$ otherwise. The \textit{depth}
of $S/I$, denoted ${\rm depth}(S/I)$, is the largest integer $\ell$
so that there is a homogeneous 
sequence $g_1,\ldots,g_\ell$ in $\mathfrak{m}=(t_1,\ldots,t_s)$ with $g_i$ not a
zero-divisor of $S/(I,g_1,\ldots,g_{i-1})$ for all $1\leq i\leq\ell$.  

The Auslander--Buchsbaum formula \cite[p.~226]{monalg-rev}:
\begin{equation}\label{AB}
{\rm pd}_S(S/I) + {\rm depth}(S/I) = \dim(S)=s
\end{equation}
relates the depth and the projective dimension. 

\begin{theorem}\label{eagon-reiner-terai}  
Let $\mathcal{C}$ be a clutter and let $I_c(\mathcal{C})$ be its ideal
of covers. The following hold.
\begin{enumerate}
\item[\rm(a)] {\rm(Eagon-Reiner \cite{ER})} $I(\mathcal{C})$ 
is Cohen--Macaulay if and only if $I_c(\mathcal{C})$
has a linear resolution.
\item[\rm(b)] {\rm(Terai \cite{terai})} 
${\rm reg}(I(\mathcal{C}))=1+{\rm reg}(S/I(\mathcal{C}))={\rm pd}(
S/I_c(\mathcal{C}))$.
\end{enumerate}
\end{theorem}

\begin{proposition}\label{linear-reg}
Let $\mathcal{C}$ be a clutter, let $I(\mathcal{C})$ be its edge 
ideal and let $d=\alpha(I(\mathcal{C}))$ be the least degree of a minimal
generator of $I(\mathcal{C})$. Then, the edge ideal $I(\mathcal{C})$ has a $d$-linear
resolution if and only if ${\rm reg}(S/I(\mathcal{C}))=d-1$.
\end{proposition}

\begin{proof} Note that, by duality \cite[p.~222]{monalg-rev}, 
$d=\alpha(I(\mathcal{C}))={\rm
ht}(I_c(\mathcal{C}))=\alpha_0(\mathcal{C}^\vee)$. 

$\Rightarrow$) Assume that $I(\mathcal{C})$ has a $d$-linear
resolution. Then, by Theorem~\ref{eagon-reiner-terai}(a),
$I_c(\mathcal{C})$ is Cohen--Macaulay, that is,
$$
{\rm pd}(S/I_c(\mathcal{C}))={\rm ht}(I_c(\mathcal{C}))=d.
$$
\quad Hence, by Theorem~\ref{eagon-reiner-terai}(b),
$1+{\rm reg}(S/I(\mathcal{C}))={\rm pd}(
S/I_c(\mathcal{C}))=d$, and ${\rm reg}(S/I(\mathcal{C}))=d-1$.

$\Leftarrow$) Assume that ${\rm reg}(S/I(\mathcal{C}))=d-1$. 
Then, by Theorem~\ref{eagon-reiner-terai}(b), one has
$${\rm pd}(
S/I_c(\mathcal{C}))=1+{\rm reg}(S/I(\mathcal{C}))=d={\rm
ht}(I_c(\mathcal{C})).$$
\quad Thus, $I_c(\mathcal{C})$ is Cohen--Macaulay and, by 
Theorem~\ref{eagon-reiner-terai}(a), $I(\mathcal{C})$ has a $d$-linear
resolution.
\end{proof}

\begin{theorem}\label{konig-theorem}\cite[Theorem~10.2]{Har} 
If $G$ is a bipartite graph and $\alpha_0(G)$ is its covering number,
then $G$ has a matching with
$\alpha_0(G)$ edges. 
\end{theorem}

\begin{proposition}\label{unmixed-Koning-CM} 
Let $G$ be a K\"onig graph without isolated vertices. The following
hold. 
\begin{enumerate}
\item[\rm(a)] \cite[Proposition~15]{Ivan-Cruz-Reyes} $G$ is unmixed if
and only if $G$ has a perfect matching $P$ of  K\"onig type 
with the {\rm (P)} property.
\item[\rm(b)] \cite[Proposition~28]{Ivan-Cruz-Reyes} 
$G$ is Cohen--Macaulay if and only if $G$ has a perfect matching of  K\"onig type 
$P$ with the {\rm (P)} property such that there are no  
$4$-cycles with two edges in $P$.
\end{enumerate}
\end{proposition}

\begin{definition}\label{vwc}
An unmixed graph $G$ without isolated vertices is 
called \textit{very well-covered} if $2\alpha_0(G)=|V (G)|$.
\end{definition}

\begin{proposition}\label{vw-c} A graph $G$ is very well-covered if
and only if $G$ is unmixed and has a perfect matching $P$ of K\"onig
type.
\end{proposition}

\begin{proof} This result follows readily from the fact that very
well-covered graphs have a perfect matching \cite[Lemma~2.1]{Kimura-Terai}.
\end{proof}

\begin{proposition}\cite[Corollary~29]{Ivan-Cruz-Reyes}\label{CM-free-vertex} 
Let $G$ be a Cohen--Macaulay K\"onig graph without
isolated vertices, then $G$ has at least one pendant vertex.
\end{proposition}

In what follows $I$ is a monomial ideal of $S$ minimally generated 
by $\{t^{v_i}\}_{i=1}^q$. We say $I$ 
is \textit{linearly presented} if $\deg(t^{v_i})=d$ for all $i$ and 
the kernel of the presentation map $\psi\colon S^q\rightarrow S$, 
$e_i\mapsto t^{v_i}$, is generated by vectors whose entries 
are linear forms, where $e_i$ is the $i$-th unit vector in $S^q$. 

If 
$I$ is generated in degree $d$, the \textit{graph} of $I$, 
denoted $\mathcal{G}_I$, has 
vertex set $V(\mathcal{G}_I):=\{t^{v_i}\}_{i=1}^q$, 
and $\{u,v\}$ is an edge of $\mathcal{G}_I$ if and only if  
$u\neq v$ and $\deg({\rm lcm}(u,v))= d+1$. 
For all $u,v$ in $V(\mathcal{G}_I)$, $u \neq v$, let
$\mathcal{G}_I^{(u,v)}$ be the induced subgraph of $\mathcal{G}_I$  
with vertex set
$$
V(\mathcal{G}_I^{(u,v)})=\{t^{v_i}\mid 
t^{v_i}\mbox{ divides } {\rm lcm}(u, v)\}.
$$

\begin{proposition}\cite[Proposition~2.1,
Corollary~2.2]{Bigdeli-etal}\label{Bigdeli-Herzog}  
Let $I$ be a monomial ideal of $S$ generated in degree $d$. The
following conditions are equivalent.
\begin{enumerate}
\item[\rm(a)] The ideal $I$ is linearly
presented. 
\item[\rm(b)] $\mathcal{G}_I^{(u,v)}$ is connected for all
$u,v\in\{t^{v_i}\}_{i=1}^q$.
\item[\rm(c)]
For all $u, v\in\{t^{v_i}\}_{i=1}^q$ there is a path in $\mathcal{G}_I^{(u,v)}$ 
connecting $u$ and $v$ whose vertices are minimal generators of $I$ dividing 
${\rm lcm}(u,v)$.
\end{enumerate}
\end{proposition}

Consider the following basic matrices \cite{birational}:
\begin{enumerate}
\item[\rm(a)] The 
\textit{syzygy matrix} of $I$, denoted
$\mathcal{S}(I)$, is the $q\times\binom{q}{2}$ matrix whose columns 
are the vectors of the form
$$
({{\rm lcm}(t^{v_i},t^{v_j})}/{t^{v_i}})e_i-({{\rm
lcm}(t^{v_i},t^{v_j})}/{t^{v_j}})e_j.
$$ 
Each column of $\mathcal{S}(I)$ defines a 
syzygy of $I$ because 
$(t^{v_1},\ldots t^{v_q})\mathcal{S}(I)=(0)$. 
\item[\rm(b)] The 
\textit{linear syzygy matrix} of $I$, denoted
$\mathcal{L}\mathcal{S}(I)$, is the $q\times|E(\mathcal{G}_I)|$ matrix whose columns are the vectors of the form
$t_\ell e_i-t_ke_j$ such that $t_\ell t^{v_i}=t_k{t^{v_j}}$. We set
$\mathcal{LS}(I)=(0)$ if $I$ has no linear syzygies.  
Note that $\mathcal{L}\mathcal{S}(I)$ is a submatrix of 
$\mathcal{S}(I)$. 
\item[\rm (c)] The \textit{numerical linear
syzygy matrix}, denoted by $\mathcal{N\!LS}(I)$, is the $\{0,\pm 1\}$-matrix 
obtained from  $\mathcal{L}\mathcal{S}(I)$ by making the
substitution $t_i=1$ for all $i$. This matrix has entries in
$\{0,\pm 1\}\subset \mathbb{Z}$ if 
the base field $K$ has characteristic $0$.  

\item[\rm(d)] The \textit{oriented graph} of $\mathcal{G}_I$ obtained 
from $\mathcal{G}_I$ by fixing an orientation for each edge of 
$\mathcal{G}_I$ is denoted by $\mathcal{D}_I$. 
\end{enumerate}

\begin{proposition}\label{syz} Let $I$ be a monomial ideal minimally generated by
$t^{v_1},\ldots,t^{v_q}$ and let $\mathcal{S}(I)$ be the syzygy
matrix 
of $I$. The following hold.
\begin{enumerate}
\item[\rm(a)] \cite[Theorem~3.4.1]{AL} The kernel of the presentation map 
$\psi\colon S^q\rightarrow S$, $e_i\mapsto t^{v_i}$, is generated by 
the columns of $\mathcal{S}(I)$.
\item[\rm(b)] Let $\varphi$ be a $q\times m$ submatrix of
$\mathcal{S}(I)$ whose columns minimally generate a graded $S$-module.
Then, ${\rm im}(\varphi)=\ker(\psi)$
if and only if ${\rm rank}(\varphi)=q-1$ 
and ${\rm ht}(I_{q-1}(\varphi))\geq 2$. 
\item[\rm(c)] If $I$ is generated by monomials of degree $d$, then 
$I$ is linearly presented if and only ${\rm rank}(\mathcal{L}\mathcal{S}(I))=q-1$ 
and ${\rm ht}(I_{q-1}(\mathcal{L}\mathcal{S}(I)))\geq 2$. 
\end{enumerate}
\end{proposition}

\begin{proof} (b) As $\varphi$ is a submatrix of the syzygy matrix
of $I$, letting
$\psi=(t^{v_1},\ldots,t^{v_q})$, the sequence 
$$
S^m\stackrel{\varphi}{\longrightarrow}
S^q\stackrel{\psi}{\longrightarrow}S
$$
is a free complex of graded $S$-modules. Then, by the 
Buchsbaum--Eisenbud acyclicity criterion
\cite[Corollary~1]{BE}, it is seen that ${\rm im}(\varphi)=\ker(\psi)$
if and only if  
$$q={\rm rank}(\varphi)+{\rm
rank}(\psi)={\rm rank}(\varphi)+1 \mbox{ and }\ 
{\rm ht}(I_{q-1}(\varphi))\geq 2. 
$$
\quad (c) This part follows from (a) and (b).  
\end{proof}

\begin{lemma}\cite[cf.~Lemma~3.5]{birational}\label{chicapitanga}
Let $\mathcal{S}(I)$ be the syzygy matrix of a monomial ideal $I$. Then, any nonzero minor of
$\mathcal{S}(I)$ is a monomial with coefficient $\pm 1$, i.e.,
it has the form $\pm t^a$.
\end{lemma}

\begin{proposition}\label{apr14-05} Let $I$ be a monomial ideal generated in
degree $d$ and let 
$c$ be the number of connected
components of the graph $\mathcal{G}_I$ of $I$. Then   
$${\rm rank}(\mathcal{LS}(I))={\rm rank}(\mathcal{N\!LS}(I))=q-c.
$$
\quad In particular, $\mathcal{G}_I$ is connected if and only if 
${\rm rank}(\mathcal{LS}(I))=q-1$. 
\end{proposition}

\begin{proof} By Lemma~\ref{chicapitanga}, we may assume
that $K=\mathbb{Q}$. Let $r_0={\rm rank}(\mathcal{LS}(I))$ and let
$M$ be a $r_0\times r_0$ submatrix of $\mathcal{LS}(I)$ with
$\det(M)\neq 0$. Then, by Lemma~\ref{chicapitanga}, $\det(M)=\pm t^a$
for some $t^a\in S$. Thus, making $t_i=1$ for all $i$, we get 
${\rm rank}(\mathcal{LS}(I))\leq {\rm rank}(\mathcal{N\!LS}(I))$. The
reverse inequality follows by definition of $\mathcal{N\!LS}(I)$. 
Let $\mathcal{D}_I$ be the digraph obtained from $\mathcal{G}_I$ by choosing an orientation for the 
edges of $\mathcal{G}_I$. The incidence matrix of $\mathcal{D}_I$ 
is equal to $\mathcal{N\!LS}(I)$ and this is a totally unimodular
matrix. Hence,
by \cite[Theorem~8.3.1]{godsil}, the rank of $\mathcal{N\!LS}(I)$ is
$q-c$. 
\end{proof}

\begin{lemma}{\rm(Inclusion-exclusion \cite[p.~38]{aigner})}\label{inc-exc}
Let $A_1,\ldots,A_n$ be finite subsets of $S$,
then
\begin{small}
\begin{equation*}
\left|\bigcup_{i=1}^n
A_i\right|=\sum_{i=1}^n|A_i|-\sum_{i<j}|A_i\cap A_j|+
\sum_{i<j<k}|A_i\cap A_j\cap
A_k|\mp\cdots+(-1)^{n-1}\left|\bigcap_{i=1}^n A_i\right|.
\end{equation*}
\end{small}
\end{lemma}

\section{The v-number of ideals of covers}\label{section-v-number-d}
In what follows to avoid repetitions, we continue to employ
the notations and
definitions used in Sections~\ref{intro-section} and
\ref{prelim-section}.

\begin{lemma}\label{jul13-24} Let $G$ be a graph with $E(G)\neq\emptyset$. Then, there
exists a vertex cover $C$ of $G$ such that $C$ has the exchange
property and $|C|=|V(G)|-1$. 
\end{lemma}

\begin{proof} Pick an edge $\{t_k,t_\ell\}$ of $G$ and let $
A=V(G)\setminus\{t_k,t_\ell\}$. Then, $C=A\cup\{t_k\}$ is a vertex
cover of $G$ such that $(C\setminus\{t_k\})\cup\{t_\ell\}$ is a vertex
cover of $G$ and $|C|=s-1=|V(G)|-1$.
\end{proof}

\begin{proposition}{\rm(Exchange property)}\label{equiv-exchange} 
Let $C$ be a vertex cover of a graph $G$.
The following conditions are equivalent.
\begin{enumerate}
\item[\rm(a)] There exists $\{t_k,t_\ell\}\in E(G)$
such that $t_k\in C$, $t_\ell\notin C$, and $N_G(t_k)\setminus
C=\{t_\ell\}$.
\item[\rm(b)] There exists $\{t_k,t_\ell\}\in E(G)$
such that $t_k\in C$, $t_\ell\notin C$, and
$(C\setminus\{t_k\})\cup\{t_\ell\}$ is a vertex cover of $G$.
\item[\rm(c)] There exist 
$\{t_k,t_\ell\}\in E(G)$ and a vertex cover $C_1$ of $G$ 
such that $C\triangle C_1=\{t_k,t_\ell\}$, 
where $C\triangle C_1:=(C\setminus C_1)\cup(C_1\setminus
C)$ is the symmetric difference.
\end{enumerate}
\end{proposition}

\begin{proof} (a) $\Rightarrow$ (b) Letting 
$C_1=(C\setminus\{t_k\})\cup\{t_\ell\}$, we show that $C_1$ is a
vertex cover of $G$. Take $e=\{t_i,t_j\}$. We may assume 
$t_\ell\notin e$, and $t_i\in C$ because $C$ is a vertex 
cover of $G$. Hence, if $t_i\neq t_k$, then $e\cap C_1\neq\emptyset$.
Thus, we may assume $t_i=t_k$. Then, $t_j\in N_G(t_k)$, and
consequently $t_j\in C$ 
because $N_G(t_k)\setminus C=\{t_\ell\}$. Hence, $t_j\in
C\setminus\{t_k\}$ because $t_i\neq t_j$. Thus, $e\cap
C_1\neq\emptyset$.

(b) $\Rightarrow$ (c) Letting 
$C_1=(C\setminus\{t_k\})\cup\{t_\ell\}$, one has $C\setminus
C_1=\{t_k\}$ and $C_1\setminus C=\{t_\ell\}$. Thus, the symmetric
difference $C\triangle C_1$ is equal to $\{t_k,t_\ell\}$.

(c) $\Rightarrow$ (a) If $t_k\in C\setminus C_1$ (resp. $t_\ell\in
C\setminus C_1$), then $t_\ell\in C_1\setminus C$ (resp. $t_k\in
C_1\setminus C$) because $C_1$ (resp. $C$) is a vertex cover of $G$.
Thus, permuting $t_k$ and $t_\ell$ if necessary, we may assume that
$t_k\in C\setminus C_1$. Clearly, 
$\{t_\ell\}\subset N_G(t_k)\setminus C$. To show the other inclusion
take $t_i\in N_G(t_k)\setminus C$. Then, $\{t_i,t_k\}\in E(G)$, and
$t_i\in C_1$ because $t_k\notin C_1$ and $C_1$ is a vertex cover of
$G$. Thus, $t_i\in C_1\setminus C$, and $t_i=t_\ell$ because 
$t_i\neq t_k$.
\end{proof}

\begin{lemma}\label{v-number-dual} Let $I_c(G)$ be the ideal of
covers of a graph $G$ and let $\alpha_e(G)$ be its exchange number. The following hold.
\begin{enumerate} 
\item[\rm(a)] {\rm(cf.~\cite[Theorem~3.5]{v-number}, \cite[Remark~2.5]{Saha})} 
${\rm v}(I_c(G))=\alpha_e(G)-1$. 
\item[\rm(b)] \cite[Lemma~3.16]{v-number} ${\rm
v}(I_c(G))\geq\alpha_0(G)-1$.
\item[\rm(c)] ${\rm v}(I_c(G))=\alpha_0(G)-1$ 
if and only if there exists a vertex cover $C$ of $G$ with 
$\alpha_0(G)$ vertices satisfying the exchange property.
\end{enumerate}
\end{lemma}

\begin{proof} (a) By Theorem~\ref{v-number-clutters-graphs}, 
there exist an associated prime $(t_k,t_\ell)$ of $I_c(G)$,
$\{t_k,t_\ell\}\in E(G)$, and a set of vertices $A\subset V(G)$ such that 
$$
(I_c(G)\colon t_A)=(t_k,\,t_\ell)
$$
and ${\rm v}(I_c(G))=\deg(t_A)$, where $t_A=\prod_{t_i\in A}t_i$.
Then, $t_kt_A$ and $t_{\ell}t_A$ are in $I_c(G)$, and consequently 
$C=A\cup\{t_k\}$ and $C_1=A\cup\{t_\ell\}$ are vertex covers of $G$.
Note that $\{t_k,t_\ell\}\cap A=\emptyset$ because $I_c(G)$ is
 a squarefree monomial ideal and $t_A\notin I_c(G)$. 
Hence, $C\setminus C_1=\{t_k\}$ and $C_1\setminus C=\{t_\ell\}$
and, by Proposition~\ref{equiv-exchange}, $C$ has the exchange
property. Thus, 
$$
{\rm v}(I_c(G))=\deg(t_A)=|C|-1\geq \alpha_e(G)-1,
$$
and consequently ${\rm v}(I_c(G))\geq \alpha_e(G)-1$. To show the
reverse inequality note that, by definition of $\alpha_e(G)$, there is a vertex
cover $C$ of $G$ satisfying the exchange property and
$\alpha_e(G)=|C|$. Hence, there exists $\{t_k,t_\ell\}\in E(G)$
such that $t_k\in C$, $t_\ell\notin C$, and
$(C\setminus\{t_k\})\cup\{t_\ell\}$ is a vertex cover of $G$. 
Letting $A:=C\setminus\{t_k\}$, the sets 
$A\cup\{t_k\}$ and $A\cup\{t_\ell\}$ are minimal vertex covers of $G$. Thus,
$t_kt_A\in I_c(G)$ and $t_\ell t_A\in I_c(G)$, that is, 
$\{t_k,t_\ell\}\subset(I_c(G)\colon t_A)$. From the embedding 
$$
S/(I_c(G)\colon t_A)\xhookrightarrow[]{\ \ t_A\ \ }S/I_c(G),
$$
and observing that any associated prime $\mathfrak{p}$ of $I_c(G)$ is of the form
$\mathfrak{p}=(e)$ for some $e\in E(G)$ it follows that $(t_k,t_\ell)$ is an associated
prime of $(I_c(G)\colon t_A)$ and this is the only associated prime
of $(I_c(G)\colon t_A)$. Thus, $(I_c(G)\colon t_A)=(t_k,t_\ell)$,
$$\alpha_e(G)=|C|=\deg(t_A)+1\geq {\rm v}(I_c(G))+1,$$
and consequently $\alpha_e(G)-1\geq {\rm v}(I_c(G))$.

(b) This follows from part (a) noticing
that $\alpha_e(G)\geq\alpha_0(G)$.

(c) The result follows from part (a) by 
noticing that $\alpha_e(C)=\alpha_0(G)$ if and only if there 
is a vertex cover $C$ of $G$ with $\alpha_0(G)$ vertices 
satisfying the exchange property. 
\end{proof}

Let $J:=I_c(G)$ be the ideal of covers of an
unmixed graph
$G$, let $E(G^\vee)=\{C_i\}_{i=1}^r$, and let $\mathcal{G}_J$ be the
graph of $J$. From Section~\ref{prelim-section}, recall 
that $\{C_i,C_j\}$ is an edge of $\mathcal{G}_J$ if 
$|C_i\cap C_j|=|C_i|-1$. 
For all $C_i,C_j$ in $V(\mathcal{G}_J)$, $i\neq j$, let
$\mathcal{G}_J^{(C_i,C_j)}$ denote the induced subgraph of $\mathcal{G}_J$ 
with vertex set
$$
V(\mathcal{G}_J^{(C_i,C_j)})=\{C_k \mid 
C_k\subset C_i\cup C_j\}.
$$

\begin{proposition}\label{Bigdeli-Herzog-j} Let $G$ be an unmixed graph. The
following conditions are equivalent.
\begin{enumerate}
\item[\rm(a)] The ideal $J$ is linearly
presented. 
\item[\rm(b)] $\mathcal{G}_J^{(C_i,C_j)}$ is connected for all $C_i,C_j
\in E(G^\vee)$, $i\neq j$. 
\item[\rm(c)]
For all $C_i, C_j\in E(G^\vee)$, $i\neq j$, there is a path in $\mathcal{G}_J^{(C_i,C_j)}$ 
connecting $C_i$ and $C_j$ whose vertices are contained in $C_i\cup C_j$.
\end{enumerate}
\end{proposition}

\begin{proof} It follows readily from Proposition~\ref{Bigdeli-Herzog}.
\end{proof}

\begin{theorem}\label{eclipse-apr8-24} Let $G$ be an unmixed graph, let
$C_1,\ldots,C_r$ be the minimal vertex covers of $G$, and let
$\mathcal{G}_J$ be the graph of $J=I_c(G)$. The following hold.  
\begin{enumerate}
\item[\rm(a)] $\{C_i,C_j\}$ is an edge of $\mathcal{G}_J$ if and only
if there exists $\{t_k,t_\ell\}\in E(G)$ such that $t_k\in C_i$,
$t_\ell\notin C_i$ and $(C_i\setminus\{t_k\})\cup\{t_\ell\}=C_j$. In
particular, non-isolated vertices of $\mathcal{G}_J$ have the 
exchange property. 
\item[\rm(b)] $\mathcal{G}_J$ is not a discrete graph if and only if 
${\rm v}(I_c(G))=\alpha_0(G)-1$.
\item[\rm(c)] If $I_c(G)$ is linearly presented, then all 
minimal vertex covers of $G$ have the exchange property and 
${\rm v}(I_c(G))=\alpha_0(G)-1$.
\end{enumerate}
\end{theorem}

\begin{proof} (a) $\Rightarrow$) 
Let $\{C_i,C_j\}$ be an edge of $\mathcal{G}_J$
and let $d=\alpha_0(G)$. Then, by definition of $\mathcal{G}_J$, one
has $|C_i\cap C_j|=|C_i|-1=|C_j|-1=d-1$. Thus, we may assume
$$ 
C_i\cap C_j=\{t_1,\ldots,t_{d-1}\},\ C_i=\{t_1,\ldots,t_{d-1},t_k\},
\ C_j=\{t_1,\ldots,t_{d-1},t_\ell\},
$$
where $t_k\neq t_\ell$ and $t_\ell\notin C_i$. Then,
$(C_i\setminus\{t_k\})\cup\{t_\ell\}=C_j$. We now show that
$\{t_k,t_\ell\}$ is an edge of $G$. 
By the minimality of
the vertex cover $C_i$, there is $e\in E(G)$ such that
$e\cap\{t_1,\ldots,t_{d-1}\}=\emptyset$. As $C_i$ and $C_j$ are vertex
covers of $G$, we get $e=\{t_k,t_\ell\}$. 

$\Leftarrow$) Let $C_i,C_j$ be vertex covers of $G$ and let 
$\{t_k,t_\ell\}$ be an edge of $G$ such that $t_k\in C_i$,
$t_\ell\notin C_i$ and $(C_i\setminus\{t_k\})\cup\{t_\ell\}=C_j$.
Then, using the inclusion-exclusion formula, from the equality
$$
(C_i\cup C_j)\setminus\{t_k,t_\ell\}=C_i\cap C_j,
$$
we get $(|C_i|+|C_j|-|C_i\cap C_j|)-2=|C_i\cap C_j|$ and since
$d=|C_i|=|C_j|$, one has 
$$|C_i\cap C_j|=d-1=|C_i|-1=|C_j|-1,$$
that is, $\{C_i,C_j\}$ is an edge of $\mathcal{G}_J$.

(b) $\Rightarrow$) Pick an edge $\{C_i,C_j\}$ of $\mathcal{G}_J$. Then, by part (a),
$C_i$ has the exchange property. Hence, by
Lemma~\ref{v-number-dual}(c), we get ${\rm v}(I_c(G))=\alpha_0(G)-1$.

$\Leftarrow$) By Lemma~\ref{v-number-dual}(c), there is a vertex cover
$C$ of $G$ with $\alpha_0(G)$ vertices satisfying the exchange
property. Then, by part (a),
$C$ is not an isolated vertex of $\mathcal{G}_J$.

(c) The graph $G$ has at least two 
minimal vertex covers because $G$ has at least one edge, 
that is, the graph $\mathcal{G}_J$ has at
least two vertices. By Proposition~\ref{Bigdeli-Herzog-j}, the graph 
$\mathcal{G}_J$ is
connected. Then, by part (a), all 
minimal vertex covers of $G$ have the exchange property. Hence, by
part (b), we obtain the equality ${\rm v}(I_c(G))=\alpha_0(G)-1$.
\end{proof}

\begin{proposition}\label{saha->exchange} 
Let $G$ be a graph. If $N_G(t_k)\subset N_G[t_\ell]$ for some
$\{t_k,t_\ell\}\in E(G)$, then there exists a minimal vertex cover $C$
of $G$ such that $N_G(t_k)\setminus
C=\{t_\ell\}$. 
\end{proposition}

\begin{proof} Pick a minimal vertex cover $C$ of $G$ that does not
contain $t_\ell$. Then, $t_k\in C$. The inclusion 
$N_G(t_k)\setminus C\supset\{t_\ell\}$ is clear. To show the reverse
inclusion take $t_i\in N_G(t_k)\setminus C$, then $t_i\in N_G(t_k)$,
and consequently $t_i\in N_G[t_\ell]$. Note that $N_G(t_\ell)\subset
C$ because $t_\ell\notin C$. Hence, if $t_i\neq t_\ell$, then $t_i\in
C$, a contradiction. Thus, $t_i=t_\ell$, as required.
\end{proof} 

\begin{corollary}\label{apr10-24} Let $G$ be a graph and let $I_c(G)$
be its ideal of covers. The following
hold.
\begin{enumerate}
\item[\rm(a)] Suppose $I(G)$ is sequentially Cohen--Macaulay. Then, 
$$
{\rm reg}(S/I_c(G))={\rm
bight}(I(G))-1\geq {\rm v}(I_c(G))=\alpha_e(G)-1\geq\alpha_0(G)-1,
$$
with equality everywhere if and only if $G$ is unmixed.
\item[\rm(b)] \cite{terai} $S/I(G)$ is Cohen--Macaulay if and only if 
${\rm reg}(S/I_c(G))=\alpha_0(G)-1$.
\item[\rm(c)] \cite[Corollary~3.9]{Saha} 
If $S/I(G)$ is Cohen--Macaulay, then 
${\rm v}(I_c(G))=\alpha_0(G)-1$.
\end{enumerate}
\end{corollary}

\begin{proof} (a) The equality ${\rm reg}(S/I_c(G))={\rm
bight}(I(G))-1$ was proved in \cite[Theorem~3.31]{edge-ideals}. 
Then, by Eq.~\eqref{summarize-ineq} and Lemma~\ref{v-number-dual}(a),
we get 
$${\rm reg}(S/I_c(G))={\rm
bight}(I(G))-1\geq {\rm v}(I_c(G))=\alpha_e(G)-1\geq
\alpha_0(G)-1,
$$
and equality holds everywhere if and only if ${\rm
bight}(I(G))=\alpha_0(G)$, and the latter equality holds if and only if
$G$ is unmixed. 

(b) By Theorem~\ref{eagon-reiner-terai}(b) and duality, one has 
${\rm
reg}(S/I_c(G))={\rm pd}_S(S/I(G))-1$. Hence, the result follows
recalling that $\alpha_0(G)$ is the height of $I(G)$. 

(c) The result follows from part (a) by 
recalling that $S/I(G)$ is Cohen--Macaulay if and only if $S/I(G)$ 
is sequentially Cohen--Macaulay and $I(G)$ is unmixed \cite[Corollary~6.3.31]{monalg-rev}.
\end{proof}

For use below, let $\mathcal{F}$ be the set of all maximal stable set 
of $G^\vee$ and let $\mathcal{A}$ be the set of all $A$ such that $A$ is a stable set 
of $G^\vee$ and $N_{G^\vee}(A)$ is a minimal vertex
cover of $G^\vee$.

\begin{proposition}\label{apr17-24} Let $G$ be a graph. 
Then, ${\rm v}(I_c(G))=\dim(S/I_c(G))$ if and only if 
$\mathcal{F}=\mathcal{A}$.
\end{proposition} 

\begin{proof} This follows by applying \cite[Corollary~3.7]{v-number}
to the clutter $G^\vee$ and noticing that $G^\vee$ is unmixed 
because the minimal covers of $G^\vee$ are precisely the edges of 
$G$.  
\end{proof}

\begin{proposition}\label{codi-graph-char} Let $G$ be a graph with $s$ vertices. 
The following conditions are
equivalent.

{\rm(a)} $G$ is a codi-graph.\quad {\rm(b)} ${\rm depth}(S/I(G))=1$.\quad
{\rm(c)} ${\rm reg}(S/I_c(G))=s-2$.
\end{proposition}

\begin{proof} (a)$\Leftrightarrow$(b) Let $\Delta$ be the independence complex of $G$
consisting of all stable sets of $G$ and let $K[\Delta]$
be its Stanley--Reisner ring. Note that $K[\Delta]=S/I(G)$. By
\cite{dsmith}, one has 
\begin{equation}\label{apr16-24}
{\rm depth}(S/I(G))=1+\max\{i\, \vert\, K[\Delta^i]
\mbox{ is Cohen--Macaulay}\},
\end{equation}
where $\Delta^i=\{F\in\Delta\, \vert\, 
\dim(F)\leq i\}$ is the $i$-skeleton of $\Delta$ and $-1\leq
i\leq\dim(\Delta)$. Note that $\Delta^1$ is the 1-skeleton of 
$\Delta^i$ for all $1\leq i\leq\dim(\Delta)$. As the facets of
$\Delta^1$ are the edges of $\overline{G}$ and $\Delta^1$ has
dimension $1$,  by
\cite[Corollary 6.3.14]{monalg-rev}, $K[\Delta^1]$ is Cohen--Macaulay
if and only if $\overline{G}$ is connected. Using the fact that the
1-skeleton of a Cohen--Macaulay complex is Cohen--Macaulay
\cite[Proposition~6.3.17]{monalg-rev}, we obtain that $K[\Delta^i]$
is not Cohen--Macaulay for all $1\leq i\leq\dim(\Delta)$ if and only 
$\Delta^1$ is not Cohen--Macaulay. Hence, by Eq.~\eqref{apr16-24}, 
${\rm depth}(S/I(G))=1$ if and only if $\overline{G}$ is disconnected.

(b)$\Leftrightarrow$(c) From Auslander--Buchsbaum and Terai's 
formulas (Eq.~\eqref{AB}, Theorem~\ref{eagon-reiner-terai}(b)): 
$${\rm reg}(S/I_c(G))={\rm pd}_S(S/I(G))-1,\quad 
{\rm pd}_S(S/I(G))+{\rm depth}(S/I(G))=s,
$$
we obtain ${\rm reg}(S/I_c(G))=s-{\rm depth}(S/I(G))-1$, and the
equivalence of (b) and (c) follows. 
\end{proof}

\begin{theorem}\label{v-ci-maximum=s-2}
Let $G$ be a graph with $s$ vertices such that $E(G)\neq\emptyset$ 
and let $I_c(G)$ be its ideal of
covers. The following conditions are equivalent.
\begin{enumerate}
\item[\rm(a)] ${\rm v}(I_c(G))=s-2$. 
\item[\rm(b)] ${\rm v}(I_c(G))={\rm
reg}(S/I_c(G))=\dim(S/I_c(G))=s-2$ and $G$ is a codi-graph.
\item[\rm(c)] Every vertex cover of $G$ with
the exchange property has exactly $s-1$ vertices.
\item[\rm(d)] ${\rm v}(I_c(H))=|V(H)|-2$ for any induced subgraph $H$ of
$G$ with $E(H)\neq\emptyset$.
\item[\rm(e)] Any induced subgraph $H$ of
$G$ with $E(H)\neq\emptyset$ is connected.
\item[\rm(f)] $G$ is a complete multipartite graph.
\end{enumerate}
\end{theorem} 

\begin{proof} (a)$\Rightarrow$(b) 
As ${\rm v}(I_c(G))=s-2$, equality holds everywhere in 
Eq.~\eqref{summarize-ineq}. That $G$ is a codi-graph follows
from Proposition~\ref{codi-graph-char}.

(b)$\Rightarrow$(c) Take any vertex cover $C$ of
$G$ with the exchange property. Then, there exists $e=\{t_k,t_\ell\}\in E(G)$
such that $t_k\in C$, $t_\ell\notin C$, and
$(C\setminus\{t_k\})\cup\{t_\ell\}$ is a vertex cover of $G$. Letting
$A=C\setminus\{t_k\}$, note that $A$ is a stable set of $G^\vee$
because $e\in E(G)$ and $e\cap A=\emptyset$. As $A\cup\{t_k\}$ and
$A\cup\{t_\ell\}$ are vertex covers of $G$, one has 
$N_{G^\vee}(A)\supset e$, and consequently $(I_c(G)\colon
t_A)\supset(e)$. Hence, from the proof of the second part of
Lemma~\ref{v-number-dual}(a), we obtain $(I_c(G)\colon
t_A)=(e)$. Therefore, $N_{G^\vee}(A)=e$ and consequently
$A\in\mathcal{A}$. Then, by Proposition~\ref{apr17-24},
$\mathcal{F}=\mathcal{A}$ and $A\in\mathcal{F}$. Thus, $A$ is a
maximal stable set of $G^\vee$ and $V(G)\setminus A$ is a minimal
vertex cover of $G^\vee$. Since the minimal vertex covers of $G^\vee$
are the edges of the graph $G$, we get 
$$
2=|V(G)\setminus A|=|V(G)|-|A|=s-(|C|-1),\ \mbox{therefore}\ |C|=s-1.
$$
\quad (c)$\Rightarrow$(a) By Lemma~\ref{v-number-dual}(a), there exists 
a vertex cover $C$ of $G$ with the exchange property such that 
${\rm v}(I_c(G))=|C|-1$. Thus, ${\rm v}(I_c(G))=s-2$.

Thus, we have proved that (a), (b), and (c) are equivalent
conditions.

(a)$\Rightarrow$(d) Let $H$ be an induced subgraph of $G$ with $E(H)\neq\emptyset$ and
let $D$ be any vertex cover of $H$ with the exchange property. Then,
there is $\{t_k,t_\ell\}\in E(H)$, $t_k\in D$, $t_\ell\notin D$, such
that $(D\setminus\{t_k\})\cup\{t_\ell\}$ is a vertex cover of $H$.
Letting $C=D\cup(V(G)\setminus V(H))$, note that $C$ is a vertex
cover of $G$. From the equality 
$$
(C\setminus\{t_k\})\cup\{t_\ell\}=((D\setminus\{t_k\})\cup\{t_\ell\})\cup
(V(G)\setminus V(H)),
$$
we obtain that $(C\setminus\{t_k\})\cup\{t_\ell\}$ is also a vertex
cover of $G$. Thus, $C$ has the exchange property
relative to $G$. 
Then, applying the equivalence of (a) and (c), we get
$s-2=|C|-1=|D|+(s-n)-1$, where $n=|V(H)|$, and consequently $|D|=n-1$.
Then, applying the equivalence of (a) and (c) to the graph $H$, 
we get that ${\rm v}(I_c(H))=n-2$.  

(d)$\Rightarrow$(e) Applying the equivalence of (a) and (b) to the
graph $H$, we obtain that the complement
$\overline{H}$ 
of $H$ is a
disconnected graph. Then, $H$ is a connected graph.

(e)$\Rightarrow$(c) Let $C$ be any vertex cover of $G$ with the
exchange property. Then, by Proposition~\ref{equiv-exchange}, 
there exists $\{t_k,t_\ell\}\in E(G)$
such that $t_k\in C$, $t_\ell\notin C$, and $N_G(t_k)\setminus
C=\{t_\ell\}$. It suffices to show that $C\cup\{t_\ell\}=V(G)$. We
argue by contradiction assuming that there is $t_i\in V(G)$ 
and $t_i\notin C\cup\{t_\ell\}$. Since the induced subgraph
$G[\{t_k,t_\ell,t_i\}]$ is connected, either $\{t_i,t_k\}\in E(G)$ 
or $\{t_i,t_\ell\}\in E(G)$. If $\{t_i,t_k\}\in E(G)$, then  
$t_i\in N_G(t_k)\setminus C=\{t_\ell\}$, a contradiction. 
If $\{t_i,t_\ell\}\in E(G)$, then $t_i\in C$ because $C$ is a vertex
cover of $G$, a contradiction.

Thus, we have proved that conditions (a) to (e) are equivalent.
Hence, to prove that (a) and (f) are equivalent, we need only show 
that (a) implies (f) and (f) implies (e).

(a)$\Rightarrow$(f) Let $H_1,\ldots,H_r$ be the connected
components of $\overline{G}$ and let $G_i=\overline{H_i}$ be its
complement for $i=1,\ldots,r$. Then, $G_1,\ldots,G_r$ are pairwise
disjoint graphs and $G=G_1*\cdots*G_r$. If $r=1$, 
by the equivalence (a) and (b), $\overline{G}=H_1$
is disconnected, a contradiction. Thus, $r\geq 2$. 
To show that $G$ is
multipartite we show that $G_i$ is a discrete graph for all $i$. 
We argue by contradiction assuming that $E(G_i)\neq\emptyset$ for some
$i$. As $G_i$ is an induced subgraph of $G$, by
applying the equivalence of (a), (d), and (b) to the graph $G_i$, we
obtain that $G_i$ is a connected codi-graph, and consequently $\overline{G_i}=H_i$
is disconnected, a contradiction.

(f)$\Rightarrow$(e) Let $G=G_1*\cdots*G_r$ be a complete multipartite
graph and let $H$ be an induced subgraph of $G$ with
$E(H)\neq\emptyset$. Pick an edge $\{t_i,t_j\}\in E(H)$. 
For any $t_k\in V(H) \setminus\{t_i,t_j\}$, one has 
that either $t_k\in N_H(t_i)$ or $t_k\in N_H(t_j)$ because $H$ is an induced
subgraph of $G$. Thus, $H$ is a connected graph.
\end{proof}

\begin{corollary}\label{vp=s-2}
If $G$ is a graph with $s$ vertices and ${\rm v}(I_c(G))=s-2$, then
\begin{enumerate}
\item[\rm(a)] $I_c(G)$ is Cohen--Macaulay,
\item[\rm(b)] ${\rm v}_{\mathfrak{p}}(I_c(G))=s-2$ for all
$\mathfrak{p}\in{\rm Ass}(I_c(G))$, and 
\item[\rm(c)] $(I_c(G)\colon\mathfrak{p})/I_c(G)$ is a principal ideal
of $S/I_c(G)$ for all $\mathfrak{p}\in{\rm Ass}(I_c(G))$.
\end{enumerate}
\end{corollary} 

\begin{proof} (a) By Theorem~\ref{v-ci-maximum=s-2}, $G$ is a
multipartite graph. Hence, $\overline{G}$ is a disjoint union of
complete graphs. Thus, $\overline{G}$ is chordal and, by a result of 
Lyubeznik \cite[Theorem 1.5]{Lyu1}, $I_c(G)$ is Cohen--Macaulay. 

(b) Let $\mathfrak{p}\in{\rm Ass}(I_c(G))$. Then, $\mathfrak{p}=(t_k,t_\ell)$ for some
$\{t_k,t_\ell\}\in E(G)$ since $I_c(G)=\bigcap_{e\in E(G)}(e)$. 
Pick $A\subset V(G)$ such that $(I_c(G)\colon
t_A)=\mathfrak{p}$ and $\deg(t_A)={\rm v}_{\mathfrak{p}}(I_c(G))$.
Note that $A\cap\{t_k,t_\ell\}=\emptyset$ because $I_c(G)$ is
squarefree. Setting $C=A\cup\{t_k\}$, it follows that $C$ is a vertex
cover of $G$ with the exchange property. Hence, by
Theorem~\ref{v-ci-maximum=s-2}, $|C|=s-1$. Thus, 
${\rm v}_{\mathfrak{p}}(I_c(G))=|A|=|C|-1=s-2$.

(c) Let $t_A+I_c(G)$, $A\subset V(G)$, be a minimal generator of 
$N:=(I_c(G)\colon\mathfrak{p})/I_c(G)$ and let $\{t_k,t_\ell\}$ be an
edge of $G$ such that $\mathfrak{p}=(t_k,t_\ell)$. 
Note that $A\cap\{t_k,t_\ell\}=\emptyset$ because $I_c(G)$ is
squarefree. Then,
$A\cup\{t_k\}$ and $A\cup\{t_\ell\}$ are vertex covers of $G$ with the
exchange property. Hence, by Theorem~\ref{v-ci-maximum=s-2}, one has 
$A\cup\{t_k,t_\ell\}=V(G)$, 
and consequently any other minimal generator of $N$ is equal to
$t_A+I_c(G)$.
\end{proof}

\begin{proposition}\label{join-formula} Let $G=G_1*G_2$ be the join of two vertex disjoint
graphs $G_1$ and $G_2$. If $G$ is not a
complete multipartite graph, then 
\begin{enumerate}
\item[\rm(a)] ${\rm v}(I_c(G))=\min\{{\rm
v}(I_c(G_1))+|V(G_2)|,\, {\rm
v}(I_c(G_2))+|V(G_1)|\}$, and 
\item[\rm(b)] ${\rm v}(I_c(G))={\rm
v}(I_c(G_1))+|V(G_2)|$ if $G_2$ is a discrete graph.
\end{enumerate}
\end{proposition}

\begin{proof} (a) 
We set $V_i=V(G_i)$, $i=1,2$, and $s=|V(G)|$. First we show the
inequalities:
\begin{equation}
\alpha_e(G)\leq 
\begin{cases}\label{apr29-24-2}
\alpha_e(G_1)+|V_2|&\mbox{if } E(G_1)\neq\emptyset,\\
\alpha_e(G_2)+|V_1|&\mbox{if } E(G_2)\neq\emptyset.
\end{cases}
\end{equation}
\quad Assume that $E(G_1)\neq\emptyset$. By Lemma~\ref{v-number-dual}, 
there is a vertex cover $D$ of $G_1$ with the exchange
property such that $|D|=\alpha_e(G_1)$.
Then, there exists $\{t_i,t_j\}\in E(G_1)$
such that $t_i\in D$, $t_j\notin D$, and
$(D\setminus\{t_i\})\cup\{t_j\}$ is a vertex cover of $G_1$. Then,  
$$
D\cup V_2\ \mbox{ and }\ ((D\cup V_2)\setminus\{t_i\})\cup\{t_j\}
$$
are vertex covers of $G$. Thus, $D\cup V_2$ is a vertex cover of
$G$ with the exchange property, and
\begin{equation*}
\alpha_e(G_1)+|V_2|=|D|+|V_2|=|D\cup V_2|\geq \alpha_e(G).
\end{equation*}
\quad If $E(G_2)\neq\emptyset$, by a similar argument, we obtain
$\alpha_e(G)\leq \alpha_e(G_2)+|V_1|$.

There is a vertex cover $C$ of $G$ with the exchange
property such that $|C|=\alpha_e(G)$ (Lemma~\ref{v-number-dual}).
Then, there exists $\{t_k,t_\ell\}\in E(G)$
such that $t_k\in C$, $t_\ell\notin C$, and
$(C\setminus\{t_k\})\cup\{t_\ell\}$ is a vertex cover of $G$. We claim
that $\{t_k,t_\ell\}\in E(G_i)$ for some $i\in\{1,2\}$. We argue by
contradiction assuming that $t_k\in V_1$ and $t_\ell\in V_2$ (the case
$t_\ell\in V_1$ and $t_k\in V_2$ can be treated similarly). If
$|C|=s-1$, then ${\rm v}(I_c(G))=s-2$ and, by
Theorem~\ref{v-ci-maximum=s-2}, $G$ is a complete multipartite graph,
a contradiction. Thus, $|C|\leq s-2$ and $C\subsetneq
V(G)\setminus\{t_\ell\}$. Pick $t_i\in V(G)\setminus\{t_\ell\}$,
$t_i\notin C$. Noticing that $V_1\subset C$ because 
$t_\ell\in V_2\setminus C$ and $G=G_1*G_2$,
we get that $t_i\notin V_1$ and $t_i\in V_2$. Then, $\{t_i,t_k\}\in E(G)$ and consequently
$t_i\in (C\setminus\{t_k\})\cup\{t_\ell\}$ because the latter is a 
vertex cover of $G$, a contradiction. This proves that
$\{t_k,t_\ell\}\in E(G_i)$ for some $i\in\{1,2\}$. 

We may assume $\{t_k,t_\ell\}\in E(G_1)$ (the case $\{t_k,t_\ell\}\in
E(G_2)$ can be treated similarly). Note that $C\supset V_2$ because
$t_\ell\notin C$ and $t_\ell\in V_1$. Then, 
$$
C\setminus V_2\ \mbox{ and }\ ((C\setminus
V_2)\setminus\{t_k\})\cup\{t_\ell\}
$$
are vertex covers of $G_1$, $C\setminus V_2$ is a vertex cover of
$G_1$ with the exchange property, and 
\begin{equation}\label{apr29-24-1}
\alpha_e(G_1)\leq|C\setminus V_2|=|C|-|V_2|=\alpha_e(G)-|V_2|.
\end{equation}
\quad Hence, from Eqs.~\eqref{apr29-24-2}-\eqref{apr29-24-1}, we
get $\alpha_e(G)=\alpha_e(G_1)+|V_2|$. Furthermore, if $E(G_2)\neq\emptyset$, by
Eq.~\eqref{apr29-24-2}, one has $\alpha_e(G)\leq \alpha_e(G_2)+|V_1|$
and 
$$\alpha_e(G)=\min\{\alpha_e(G_1)+|V_2|,\, \alpha_e(G_2)+|V_1|\}.$$
\quad Hence, by Lemma~\ref{v-number-dual}, we obtain 
$$
{\rm v}(I_c(G))=\min\{{\rm
v}(I_c(G_1))+|V_2|,\, {\rm
v}(I_c(G_2))+|V_1|\}.
$$
\quad (b) If $E(G_2)=\emptyset$, by the proof of part (a) it follows
that ${\rm v}(I_c(G))={\rm
v}(I_c(G_1))+|V_2|$.
\end{proof}

The next result is similar to \cite[Theorem~3.10]{Saha}.

\begin{corollary}\label{diff} 
Let $G_1$ be the vertex disjoint union of $k+1$
edges, let $G_2$ be a discrete
graph with one vertex, and let $G=G_1*G_2$ be their join. Then,
$$
{\rm reg}(S/I_c(G))-{\rm v}(I_c(G))=k.
$$
\end{corollary}

\begin{proof} The edge ideal $I(G_1)$ is a complete intersection,
hence it is Cohen--Macaulay. Then, by Corollary~\ref{apr10-24} and
Proposition~\ref{join-formula}, we obtain 
$$
{\rm v}(I_c(G_1))=\alpha_0(G_1)-1=k\ \mbox{  and }\ {\rm v}(I_c(G))={\rm
v}(I_c(G_1))+1=k+1.
$$
\quad As $G$ is a codi-graph, 
by Proposition~\ref{codi-graph-char}, one has ${\rm
reg}(S/I_c(G))=|V(G)|-2=2k+1$. Hence, the difference 
of the regularity of $S/I_c(G)$ and the v-number of $I_c(G)$ is equal
to $k$.
\end{proof}

\begin{proposition}\label{compl-$k$-partite}
Let $G$ be a complete $k$-partite graph. The following hold. 
\begin{enumerate}
\item[\rm(a)] If $k=2$, then  each edge of $G$ satisfies the  {\rm (P)} property.
\item[\rm(b)] If $k \geq 3$, then no edge of $G$ satisfies the  {\rm (P)} property.
\item[\rm(c)] $G$ is unmixed if and only if $G$ is a homogeneous 
multipartite graph. 
\end{enumerate}
\end{proposition}

\begin{proof} 
 We can write $G=G_1*\cdots*G_k$, where
 $G_i$ is a discrete graph for $1 \leq i \leq k$.
 
 (a) Take $e\in E(G)$, then $e=\{t_1,t_2\}$ with 
 $t_1 \in V(G_1)$ and $t_2 \in V(G_2)$ since $k=2$.
 Now, let $\{t_1,t'_1\}, \{t_2,t'_2\}$ be two edges of $G$
 distinct from $e$, then $t'_1 \in V(G_2)$ and $t'_2 \in V(G_1)$.
 Hence, $\{t'_1,t'_2\} \in E(G)$, since $G$ is a
 complete $2$-partite graph. Thus, $e$
 satisfies the  {\rm (P)} property.
 
 (b) Take $e\in E(G)$, then $e=\{t_1,t_2\}$ with 
 $t_1 \in V(G_{i_1})$ and $t_2 \in V(G_{i_2})$ where
 $1 \leq i_1 < i_2 \leq k$, since $G$ is $k$-partite.
Since  $k \geq 3$, there is $t_3 \in V(G_{i_3})$
with $i_3 \notin \{i_1,i_2\}$. Thus, 
$\{t_1,t_3\} \in E(G)$ and  $\{t_2,t_3\} \in E(G)$,
since $G$  is a complete $k$-partite graph. But 
$\{t_3,t_3\} \notin E(G)$, since $G$ is a simple
graph. Hence, $e$
does not satisfy the  {\rm (P)} property.

(c) Note that the maximal stable sets of $G$ are
$V(G_1), \ldots, V(G_k)$. Hence, $G$ is
unmixed if and only if $|V(G_1)| = \cdots = |V(G_k)|$, i.e., if and only
if $G$ is homogeneous.
 \end{proof} 

\begin{corollary}\label{jul14-24-3}  
Let $G$ be a homogeneous  complete $k$-partite 
graph. The following conditions are equivalent.
$$
{\rm(a)}\ \mathcal{G}_{I_c(G)}\mbox{ is a complete graph}. \quad
{\rm(b)}\ G \mbox{ is a complete graph}.\quad
{\rm(c)}\ {\rm v}(I_c(G))=\alpha_0(G)-1.
$$
\end{corollary}

\begin{proof} Let $s=|V(G)|$. We can write $G=G_1*\cdots* G_k$, where
$G_i$ is a discrete graph with $p$ vertices for $i=1,\ldots,k$. Then,
$G$ is unmixed by Proposition~\ref{compl-$k$-partite} and ${\rm
v}(I_c(G))=s-2$ by Theorem~\ref{v-ci-maximum=s-2}.  Then,
$\alpha_0(G)=s-p$ because all maximal stable sets of $G$
 have $p$ elements.

(b)$\Leftrightarrow$(c) The equality ${\rm v}(I_c(G))=\alpha_0(G)-1$
holds if and only if $s-2=s-p-1$, i.e., if and only if $p=1$, that is,
if and only if $G$ is a complete graph.

(a)$\Rightarrow$(c) As $\mathcal{G}_{I_c(G)}$ is connected, by 
Theorem~\ref{eclipse-apr8-24}(b), ${\rm v}(I_c(G))=\alpha_0(G)-1$.    

(b)$\Rightarrow$(a) As $G$ is complete, $G$ is unmixed and $\alpha_0(G)=s-1$.
Letting $V(G)=\{t_1,\ldots,t_s\}$ and $C_i=V(G)\setminus\{t_i\}$ for
$i=1,\ldots,s$, one has $E(G^\vee)=\{C_i\}_{i=1}^s$.
Since $|C_i\cap C_j|=|C_i|-1$ for $i\neq j$, $\{C_i,C_j\}$ is an edge 
of $\mathcal{G}_{I_c(G)}$ for $i\neq j$, and $\mathcal{G}_{I_c(G)}$
is a complete graph.
\end{proof}

\section{The graphs of ideals of covers}
 
Let $e_i, 1\leq i\leq s$, be the $i$-th unit vector of
${\mathbb R}^s$, let $C_1,\ldots,C_r$ be the minimal 
vertex covers of an unmixed graph $G$ with $d=\alpha_0(G)$, 
let $\{{t_{c_i}}\}_{i=1}^r$ be the minimal set of generators of the ideal of
covers $I_c(G)$ of $G$, where $t_{c_i}=\prod_{t_j\in C_i}t_j$, let
$\mathcal{G}=\mathcal{G}_{I_c(G)}$ be the graph of $I_c(G)$, and let $\mathcal{K}_r$ be the
complete graph with vertex set $V(\mathcal{G})$. 

\begin{proposition}\label{apr14-05-coro} Let $G$ be an unmixed graph 
and let $\mathcal{LS}(I_c(G))$ be the linear syzygy matrix of
$I_c(G)$. Then, $\mathcal{G}_{I_c(G)}$ is connected 
if and only if ${\rm rank}(\mathcal{LS}(I_c(G)))=r-1$. 
\end{proposition}

\begin{proof} Let $c$ be the number of connected
components of $\mathcal{G}_{I_c(G)}$. Then, by
Proposition~\ref{apr14-05},     
$${\rm rank}(\mathcal{LS}(I_c(G)))=r-c,
$$
and consequently $\mathcal{G}_{I_c(G)}$ is connected if and only if 
${\rm rank}(\mathcal{LS}(I_c(G)))=r-1$. 
\end{proof}

\begin{proposition}\label{jun5-24}
Let $G$ be an unmixed graph and let $\{C_1,C_2,C_3\}$ be a path of
length two of the
graph $\mathcal{G}$ of $I_c(G)$. The following hold.  
\begin{enumerate}
\item[\rm(a)]
There exist $\epsilon_i=\{t_{k_i},t_{\ell_i}\}\in E(G)$, $i=1,2$, such that 
\begin{enumerate}
\item[\rm(1)] $C_2=(C_1\setminus\{t_{k_1}\})\cup\{t_{\ell_1}\}=(C_1\cap
C_2)\cup\{t_{\ell_1}\}$, $t_{k_1}\in C_1$, $t_{\ell_1}\notin C_1$,
\item[\rm(2)] $C_3=(C_2\setminus\{t_{k_2}\})\cup\{t_{\ell_2}\}=(C_2\cap
C_3)\cup\{t_{\ell_2}\}$, $t_{k_2}\in C_2$, $t_{\ell_2}\notin C_2$, 
\item[\rm(3)] $t_{\ell_1}\neq t_{k_2}$, and $C_2\subset C_1\cup C_3$.
\end{enumerate} 
\item[\rm(b)] If the path $\{C_1,C_2,C_3\}$ of $\mathcal{G}$ is induced, then 
$\epsilon_1\cap \epsilon_2=\emptyset$, $|C_1\cap C_2\cap C_3|=\alpha_0(G)-2$,
$|C_1\cap C_3|=\alpha_0(G)-2$, and $t_{\ell_2}\notin C_1$.
\end{enumerate}
\end{proposition}

\begin{proof} (a) The existence of $\epsilon_1$ and $\epsilon_2$ satisfying (1) and
(2) follows from Theorem~\ref{eclipse-apr8-24} and its proof. Now we
prove part (3). To prove that
$t_{\ell_1}\neq t_{k_2}$, we argue by contradiction assuming
$t_{\ell_1}=t_{k_2}$. We claim that $C_1=C_3$. Since $|C_1|=|C_3|$, it
suffices to show the inclusion $C_1\supset C_3$.  Take $t_i\in C_3$. 
Note that $t_{\ell_2}\in C_1$
because $C_1$ is a vertex cover, 
$\{t_{k_2},t_{\ell_2}\}=\{t_{\ell_1},t_{\ell_2}\}\in E(G)$ and
$t_{\ell_1}\notin C_1$. Thus, if $t_i=t_{\ell_2}$, one has $t_i\in
C_1$. If $t_i\neq t_{\ell_2}$, by (2), we get
$t_i\in C_2\setminus\{t_{k_2}\}$ and $t_i\neq t_{k_2}$. Then, by (1), $t_i\in
C_1\setminus\{t_{k_1}\}$ or $t_i=t_{\ell_1}$. In the first case
$t_i\in C_1$. In the second case $t_i=t_{\ell_1}=t_{k_2}$, a
contradiction since $t_i\neq t_{k_2}$. Thus, $t_i\in C_1$. This proves
the equality $C_1=C_3$, a contradiction because $C_1,C_2, C_3$ are
distinct. To show the inclusion $C_2\subset C_1\cup C_3$ take $t_i\in
C_2$. By (1), either $t_i\in C_1\setminus\{t_{k_1}\}\subset C_1$ or
$t_i=t_{\ell_1}$. Thus, we may assume $t_i=t_{\ell_1}$. Recalling 
that $t_{\ell_1}\neq t_{k_2}$, we get $t_i=t_{\ell_1}\in
C_2\setminus\{t_{k_2}\}\subset C_3$.

(b) To show that $\epsilon_1\cap \epsilon_2=\emptyset$ it suffices to show that the
following conditions hold. 
$$
{\rm(i)}\ t_{\ell_1}\neq t_{\ell_2}, \quad {\rm(ii)}\ t_{k_1}\neq
t_{k_2}, \quad
{\rm(iii)}\ t_{\ell_1}\neq t_{k_2}, \quad {\rm(iv)}\ t_{k_1}\neq t_{\ell_2}.  
$$
\quad (i) Assume that $t_{\ell_1}=t_{\ell_2}$.  Then, by (1), 
$t_{\ell_2}=t_{\ell_1}\in C_2$, a contradiction because by the last
part of (2) 
we have $t_{\ell_2}\notin C_2$.

(ii) Assume that $t_{k_1}=t_{k_2}$. Then, by (1)-(2),
$t_{k_1}=t_{k_2}\in C_1\cap C_2=C_1\setminus\{t_{k_1}\}$, a
contradiction. 

(iii) This follows at once from part (a) item (3).

(iv) Assume that $t_{k_1}= t_{\ell_2}$. Note that $t_{k_2}\in C_1\cap
C_2$. Indeed, by (2), $t_{k_2}\in C_2$ and, by (3),  
$t_{k_2}\neq t_{\ell_1}$. Then, by (1),
$t_{k_2}\in C_1\setminus\{t_{k_1}\}\subset C_1$. From (2), we
obtain
\begin{align*}
C_1\cap C_3&=C_1\cap((C_2\setminus\{t_{k_2}\})\cup\{t_{\ell_2}\})
=(C_1\cap(C_2\setminus\{t_{k_2}\}))\cup (C_1\cap \{t_{\ell_2}\})\\
&=((C_1\cap C_2)\setminus\{t_{k_2}\})\cup (C_1\cap
\{t_{k_1}\}),\ \mbox{ therefore}\\
|C_1\cap C_3|&= |(C_1\cap C_2)\setminus\{t_{k_2}\}|+|C_1\cap
\{t_{k_1}\}|-|((C_1\cap C_2)\setminus\{t_{k_2}\})\cap (C_1\cap
\{t_{\ell_2}\})|\\
&=((\alpha_0(G)-1)-1)+1-0=\alpha_0(G)-1=|C_1|-1=|C_3|-1.
\end{align*}
\quad Hence, $\{C_1,C_3\}$ is an edge of the graph $\mathcal{G}$ of
$I_c(G)$, a contradiction
because $\{C_1,C_2,C_3\}$ is an induced path of $\mathcal{G}$. 
Letting $d= \alpha_0(G)$, we now show that
$|\bigcap_{i=1}^3C_i|=d-2$. By (3), $C_2\subset C_1\cup C_3$. Then,
$C_1\cup C_3=C_1\cup C_2\cup C_3$ and, 
by the inclusion-exclusion principle (Lemma~\ref{inc-exc}), one has 
\begin{align*}
|C_1\cup C_3|&=2d-|C_1\cap C_3|=|C_1\cup C_2\cup C_3|=
3d-2(d-1)-|C_1\cap C_3|+|C_1\cap C_2\cap C_3|,
\end{align*}   
and consequently, cancelling out $-|C_1\cap C_3|$, we obtain 
$|C_1\cap C_2\cap C_3|=d-2$. Finally, we
show that $|C_1\cap C_3|$ is equal to $\alpha_0(G)-2$ and
$t_{\ell_2}\notin C_1$. By (2) we obtain
\begin{align*}
C_1\cap C_3&=C_1\cap((C_2\cap C_3)\cup\{t_{\ell_2}\})
=(C_1\cap C_2\cap C_3)\cup (C_1\cap \{t_{\ell_2}\}),\ \mbox{ therefore}\\
|C_1\cap C_3|&= (d-2) +|C_1\cap
\{t_{\ell_2}\}|-|C_1\cap C_2\cap C_3\cap\{t_{\ell_2}\}|=
(d-2) +|C_1\cap \{t_{\ell_2}\}|.
\end{align*}
\quad Then, $t_{\ell_2}\notin C_1$ because $\{C_1,C_3\}$ is not an
edge of $\mathcal{G}$. Thus, $|C_1\cap C_3|=d-2$.
\end{proof}

\begin{proposition}\label{jul14-24}
Let $G$ be an unmixed graph. A path $\{C_1,\ldots,C_n\}$ of the graph 
$\mathcal{G}$ of $I_c(G)$ is a path of $\mathcal{G}^{(C_1,C_n)}$ if
and only if 
$$
\alpha_0(G)=|C_k\cap C_1|+|C_k\cap C_n|-|C_1\cap C_k\cap C_n|,\,\ \forall\,
1<k<n.
$$
\end{proposition}

\begin{proof} $\Rightarrow$) If $\{C_1,\ldots,C_n\}$ is a path of 
$\mathcal{G}^{(C_1,C_n)}$, then $C_k\subset C_1\cup C_n$ for all
$1<k<n$. Hence,
$$
C_k=(C_k\cap C_1)\cup(C_k\cap C_n),\ \mbox{therefore}\ \alpha_0(G)=|C_k\cap
C_1|+|C_k\cap C_n|-|C_1\cap C_k\cap C_n|.
$$
$\Leftarrow$) By assumption $C_k$ and $(C_k\cap C_1)\cup(C_k\cap C_n)$
have the same cardinality for all $1<k<n$. Since the latter is
contained in $C_k$, one has $C_k=(C_k\cap C_1)\cup(C_k\cap C_n)$. Thus, 
$C_k\subset C_1\cup C_n$. 
\end{proof}

For use below we state the following elementary fact. 

\begin{lemma}\label{existencia-cubiertas}
If $t$ is not an isolated vertex of a graph $G$, then there are
minimal vertex covers $C$ and $C'$ of $G$  
such that $t \in C'$ and $t \notin C$.  
\end{lemma} 

\begin{proposition}\label{connected-exchange}
Let $G$ be an unmixed graph without isolated vertices. If
the graph $\mathcal{G}$ of $I_c(G)$ is 
connected, then $V(G)$ is the union of the exchange edges of $G$. 
\end{proposition} 

\begin{proof} 
Let $t \in V(G)$. By Lemma~\ref{existencia-cubiertas}, there are 
minimal vertex covers $C$ and $C'$ of $G$ such that $t\in C'\setminus C$. As
 $\mathcal{G}$ is connected,  there is a path 
$\mathcal{P}=\{C_1=C,C_2, \ldots, C_n=C'\}$ in $\mathcal{G}$. Let
$$
j=\mbox{min}\{i  \mbox{ } \vert \mbox{  }  C_i \in V(\mathcal{P})  \mbox{ and }  t \in C_i \},
$$
then $2 \leq j \leq n$ since $t \notin C_1$ and $t \in C_n$. So, 
$t \in C_j  \setminus C_{j-1}$. Furthermore, $\{C_{j-1}, C_j\}$ is an
edge of $\mathcal{G}$ and, by Theorem~\ref{eclipse-apr8-24},
$C_j=(C_{j-1} \setminus \{t'\}) \cup \{t\}$ for some $\{t,t'\}\in E(G)$.
Hence, $\{t,t'\}$ is an exchange edge of $G$. Therefore, $V(G)$
is the union of the  exchange edges.
\end{proof} 

\begin{proposition}\label{g-paths}
Let $G$ be an unmixed graph, let $\mathcal{G}$ be the graph of
$I_c(G)$, and let $\{C_1,\ldots,C_n\}$ be a path of $\mathcal{G}$. The
following conditions are equivalent.
\begin{enumerate}
\item[\rm(a)] $C_i\subset C_p\cup C_q$ for all integers $1\leq p\leq i\leq
q\leq n$.
\item[\rm(b)] $|\bigcap_{j=p}^qC_j|=\alpha_0(G)-(q-p)$ for all
integers $1\leq p\leq
q\leq n$.
\end{enumerate}
\end{proposition}

\begin{proof} (a)$\Rightarrow$(b) 
We proceed by induction on $q\geq 1$. If $q=1$, then $q=p=1$ and 
$|\bigcap_{j=p}^qC_j|=|C_1|=\alpha_0(G)=\alpha_0(G)-(q-p)$. Assume that
$q\geq 2$ and the equality
in (b) valid for $q-1$. Note that the equality in
(b) is clear if $q=p$ or $q=p+1$. The case $q=p+1$ follows by recalling 
that $\{C_p,C_{p+1}\}$ is an edge of $\mathcal{G}$, that is, 
$|C_p\cap C_{p+1}|=|C_p|-1=\alpha_0(G)-1$. Thus, we may assume that $q>p+1$.  
Consider the list
$$
C_p,\, C_{p+1},\ldots,C_{q-2},\,C_{q-1},\, C_q.
$$
\quad By hypothesis $C_{q-1}\subset C_p\cup C_q,\, C_{q-1}\subset
C_{p+1}\cup C_q,\ldots,C_{q-1}\subset
C_{q-2}\cup C_q$. Hence,
$$
C_{q-1}=(C_{q-1}\cap C_q)\cup(C_p\cap C_{p+1}\cap\cdots\cap
C_{q-2}\cap C_{q-1}).
$$
\quad Then, setting $d=\alpha_0(G)$, by induction and applying
inclusion-exclusion, 
one has
$$
d=(d-1)+\Bigg|\bigcap_{j=p}^{q-1}C_j
\Big|-\Bigg|\bigcap_{j=p}^{q}C_j\Bigg|=(d-1)+(d-(q-1-p))
-\Bigg|\bigcap_{j=p}^{q}C_j\Bigg|.
$$
\quad Thus, $|\bigcap_{j=p}^{q}C_j|=d-(q-p)$ 
and the proof is
complete.

(b)$\Rightarrow$(a) We proceed by induction on $i$. 
The inclusion $C_i\subset C_p\cup C_q$ is clear
if $i=p$ or $i=q$. Assume that $p<i<q$ and $C_{i-1}
\subset C_p\cup C_q$. From the inclusion 
$$
C_{i}\supset(C_{i-1}\cap C_i)\cup(C_i\cap C_{i+1}\cap\cdots\cap
C_{q}),
$$
and the equalities $\big|\bigcap_{j=i}^{q}C_j
\big|=d-(q-i)$, $\big|\bigcap_{j=i-1}^{q}C_j\big|=d-(q-(i-1))$, and 
$$
|C_i|-|(C_{i-1}\cap C_i)\cup(C_i\cap C_{i+1}\cap\cdots\cap
C_{q})|=d-\Bigg((d-1)+\Bigg|\bigcap_{j=i}^{q}C_j
\Big|-\Bigg|\bigcap_{j=i-1}^{q}C_j\Bigg|\Bigg)=0,
$$
one obtains $C_{i}=(C_{i-1}\cap C_i)\cup(C_i\cap C_{i+1}\cap\cdots\cap
C_{q})\subset (C_p\cup C_q)\cup C_q=C_p\cup C_q$.  
\end{proof}

\begin{corollary}\label{linearp-trees}
Let $G$ be an unmixed graph. If $\mathcal{G}$ is a tree, then $I_c(G)$ has a linear
presentation if and only if for every path $\{C_1,\ldots,C_n\}$
of $\mathcal{G}$ one has $\alpha_0(G)\geq n-1$ and  
$$
\Big|\bigcap_{j=1}^nC_j\Big|=\alpha_0(G)-(n-1).
$$
\end{corollary}

\begin{proof} $\Rightarrow$) Let $\{C_1,\dots,C_n\}$ be a path of
$\mathcal{G}$. By Proposition~\ref{Bigdeli-Herzog-j}, there is a path
$\mathcal{P}$ of $\mathcal{G}$ joining $C_1$ and $C_n$ whose vertices
are in $V(\mathcal{G}^{(C_1,C_n)})$. As $\mathcal{G}$ is a tree, there
is a unique path between $C_1$ and $C_n$. Thus,
$\mathcal{P}=\{C_1,\dots,C_n\}$, and consequently 
$C_i\subset C_1\cup C_n$ for any $1\leq i\leq n$. Given integers
$1\leq p\leq i\leq q\leq n$, as $\mathcal{G}$ is a tree, 
the unique path between $C_p$ and $C_q$ 
is $\{C_p,\ldots,C_q\}$. Again, by Proposition~\ref{Bigdeli-Herzog-j},
one has $C_i\subset C_p\cup C_q$ for any $p\leq i\leq q$. Hence, by
Proposition~\ref{g-paths}, 
$|\bigcap_{j=1}^nC_j|=\alpha_0(G)-(n-1)$ and the proof is
complete.

$\Leftarrow$) Let $C_1$ and $C_n$ be two vertices of
$\mathcal{G}$. As $\mathcal{G}$ is connected, there is a path 
$\{C_1,\dots,C_n\}$ of $\mathcal{G}$ joining $C_1$ and $C_n$. Let 
$p$ and $q$ be any integers such that $1\leq p\leq q\leq n$. Then, 
$\{C_p,\ldots,C_q\}$ is a path of $\mathcal{G}$ of length $q-p$ 
and, 
by hypothesis, $|\bigcap_{j=p}^qC_j|=\alpha_0(G)-(q-p)$. Then, by 
Proposition~\ref{g-paths}, $C_i\subset C_p\cup C_q$ for $1\leq p\leq
i\leq q\leq n$. Then, in particular, $\mathcal{P}$ is a path in 
$\mathcal{G}^{(C_1,C_n)}$ and, by Proposition~\ref{Bigdeli-Herzog-j}, $I_c(G)$ has
a linear presentation.
\end{proof}

\begin{proposition}\label{jul14-24-1}
Let $G$ be an unmixed graph and let $\mathcal{G}$ be the graph of
the ideal of covers $I_c(G)$ of $G$. If $\mathcal{G}=\{C_1,\ldots,C_n\}$ is a path of
$\mathcal{G}$ of length
$n-1$, then $I_c(G)$ has a linear
presentation if and only if there exists a set of exchange edges
$\{\{t_{k_i},t_{\ell_i}\}\}_{i=1}^{n-1}$ of $G$ such that 
\begin{enumerate}
\item[\rm(a)]
$C_i=(C_{i-1}\setminus\{t_{k_{i-1}}\})\cup\{t_{\ell_{i-1}}\}$, 
$t_{k_{i-1}}\in C_{i-1}$, 
$t_{\ell_{i-1}}\notin C_{i-1}$ for $i=2,\ldots,n$, 
\item[\rm(b)] $|\bigcap_{j=p}^qC_j|=\alpha_0(G)-(q-p)$ 
for all $1\leq p\leq q\leq n$,
\item[\rm(c)] $C_1=\{t_{k_i}\}_{i=1}^{n-1}$, 
$C_n=\{t_{\ell_i}\}_{i=1}^{n-1}$, $V(G)=C_1\cup C_n$, and 
$|C_1\cap C_n|=2\alpha_0(G)-|V(G)|$.
\end{enumerate}
\end{proposition}

\begin{proof} $\Rightarrow$) By Theorem~\ref{eclipse-apr8-24}, 
there exists a set of exchange edges of $G$ that satisfies (a). As the
path $\mathcal{G}$ is a tree, by Corollary~\ref{linearp-trees}, one
obtains condition (b).  To prove (c) note that
$\bigcap_{i=1}^nC_i=\emptyset$ because the edge ideal of $G$ is given
by $I(G)=\bigcap_{i=1}^n(C_i)$ and $I(G)$ contains no variables. Then, 
making $p=1$ and $q=n$ in (b), we get that $\alpha_0(G)=n-1$. Hence to
show the equalities $C_1=\{t_{k_i}\}_{i=1}^{n-1}$ and  
$C_n=\{t_{\ell_i}\}_{i=1}^{n-1}$ it suffices to show the inclusions 
$C_1\subset \{t_{k_i}\}_{i=1}^{n-1}$ and
$C_n\subset\{t_{\ell_i}\}_{i=1}^{n-1}$. To show the first inclusion
take $t\in C_1$. If $t\notin \{t_{k_i}\}_{i=1}^{n-1}$, by part
(a), one has 
$$ 
t\in C_1\setminus\{t_{k_1}\}\subset C_2,\ldots,
t\in C_{n-1}\setminus\{t_{k_{n-1}}\}\subset C_n,
$$
and consequently $t\in\bigcap_{i=1}^nC_i$, a contradiction. Thus, 
$t\in\{t_{k_i}\}_{i=1}^{n-1}$. To show the second inclusion
take $t'\in C_n$. If $t'\notin \{t_{\ell_i}\}_{i=1}^{n-1}$, by part
(a), one has 
$$ 
t'\in C_{n-1}\setminus\{t_{k_{n-1}}\}\subset C_{n-1},\ldots,
t'\in C_{2}\setminus\{t_{k_{2}}\}\subset C_2,\,
t'\in C_{1}\setminus\{t_{k_{1}}\}\subset C_1,
$$
and consequently $t'\in\bigcap_{i=1}^nC_i$, a contradiction. Thus, 
$t'\in\{t_{\ell_i}\}_{i=1}^{n-1}$. As $\mathcal{G}$ is a path, the set of exchange
edges of $G$ is $\{\{t_{k_i},t_{\ell_i}\}\}_{i=1}^{n-1}$. Hence, 
by Proposition~\ref{connected-exchange}, the exchange edges cover 
$V(G)$, that is, $V(G)=C_1\cup C_n$. Therefore,
$|V(G)|=2\alpha_0(G)-|C_1\cap C_n|$.

$\Leftarrow$) As $\mathcal{G}$ is a path and condition (b) holds, 
by Corollary~\ref{linearp-trees}, the ideal of covers 
$I_c(G)$ of $G$ has a linear presentation.
\end{proof}

\begin{proposition}\label{I_c-bipartite} Let $H_r$ be the graph with vertex set
$V(H_r)=\{t_1,\ldots,t_r,t_{r+1},\ldots, t_{2r}\}$ and edge set 
$E(H_r)=\{\{t_i,t_{r+i}\}\}_{i=1}^r$. The following hold.
\begin{enumerate}
\item[\rm(a)] The graph $\mathcal{H}_r$ of $I_c(H_r)$ is a bipartite 
$r$-regular graph with $2^r$ vertices and $r2^{r-1}$ edges. 
\item[\rm(b)] If $G$ is an unmixed bipartite graph without isolated
vertices and $r=\alpha_0(G)$, then $G$ has a perfect matching
$\{\{t_i,t_{r+i}\}\}_{i=1}^r$, and the graph $\mathcal{G}$ 
of $I_c(G)$ is a subgraph of $\mathcal{H}_r$. In particular
$\mathcal{G}$ is bipartite.
\end{enumerate}
\end{proposition}
 
\begin{proof} (a) We give a recursive method to 
construct $\mathcal{H}_r$. 
We proceed by induction on $r\geq 2$. If $r=2$, then the minimal
vertex covers of $H_2$ are 
$$
C_1=\{t_1,t_2\},\ C_2=\{t_3,t_4\},\ C_3=\{t_1, t_4\},\ C_4=\{t_2,t_3\}.
$$
\quad Since the graph $\mathcal{H}_2$ is the $4$-cycle
$\{C_1,C_3,C_2,C_4,C_1\}$, the graph $\mathcal{H}_2$ is bipartite with
bipartition $\{C_1,C_2\}$, $\{C_3,C_4\}$. Assume $r>2$ and that the result 
holds for $\mathcal{H}_r$. Let $p=2^{r-1}$
and let $\{C_1,\ldots,C_p\}$, $\{C_{p+1},\ldots,C_{2p}\}$ be a bipartition
of the bipartite graph $\mathcal{H}_r$. Consider the graph $H_{r+1}$ 
obtained from $H_r$ by adding the edge $\{t_{2r+1},t_{2r+2}\}$. 
Then, $\alpha_0(H_{r+1})=r+1$. It is not hard to see that 
the graph $\mathcal{H}_{r+1}$ is bipartite with bipartition 
\begin{align*}
&\{C_1\cup\{t_{2r+2}\},\ldots,\quad C_p\cup\{t_{2r+2}\},
\ \ C_{p+1}\cup\{t_{2r+1}\},\ldots,C_{2p}\cup\{t_{2r+1}\}\},\\
&\{C_{p+1}\cup\{t_{2r+2}\},\ldots,C_{2p}\cup\{t_{2r+2}\},
\ C_{1}\cup\{t_{2r+1}\},\ldots,\quad C_{p}\cup\{t_{2r+1}\}\}.
\end{align*}
\quad To determine the edges of $\mathcal{H}_{r+1}$ consider the
following sets:
\begin{align*}
E_1&=\{\{C_i\cup\{t_{2r+2}\}, C_j\cup\{t_{2r+2}\}\}\mid
\{C_i,C_j\}\in E(\mathcal{H}_r)\},\\
E_2&=\{\{C_i\cup\{t_{2r+1}\}, C_j\cup\{t_{2r+1}\}\}\mid
\{C_i,C_j\}\in E(\mathcal{H}_r)\},\\
E_3&=\{\{C_i\cup\{t_{2r+2}\}, C_i\cup\{t_{2r+1}\}\}\mid
i=1,\ldots,2p\}.
\end{align*}
\quad Note that $E(\mathcal{H}_{r+1})=E_1\cup E_2\cup E_3$ and,  
by induction, it follows that $\mathcal{H}_{r+1}$ is a bipartite 
$(r+1)$-regular graph and 
$E(\mathcal{H}_{r+1})=r2^{r-1}+r2^{r-1}+2^r=(r+1)2^r$.

(b) Let $C$  be any vertex of $\mathcal{G}$. 
Then, as $G$ is unmixed, $C$ is a vertex cover of $G$ with
$r=\alpha_0(G)$ elements. By \cite[Theorem~1.1]{unmixed}, $G$ has a
perfect matching $\{\{t_i,t_{r+i}\}\}_{i=1}^r$. Then, $C$ is a vertex
cover of $H_r$, that is, $C$ is a vertex of $\mathcal{H}_r$. Thus, 
$V(\mathcal{G})\subset V(\mathcal{H}_r)$. Let $\{C,D\}$ be any edge of
$\mathcal{G}$. Then, $|C\cap D|=\alpha_0(G)-1$ and, since
$\alpha_0(G)=\alpha_0(H_r)=r$, $\{C,D\}$ is an edge of
$\mathcal{H}_r$. This proves that $\mathcal{G}$ is a subgraph 
of $\mathcal{H}_r$. Hence, by part (a), $\mathcal{G}$ is bipartite. 
\end{proof}

\begin{proposition} \label{min-gen-S}
Let $G$ be an unmixed graph. Then, the set $F$ of column vectors of
$\mathcal{L}\mathcal{S}(I_c(G))$ is a minimal system of generators
for $S(F)$, the $S$-module generated by $F$, if and only if
the graph $\mathcal{G}= \mathcal{G}_{I_c(G)}$  has no strong  
$3$-cycles.
\end{proposition}
\begin{proof} 
Let $\{C_1, \ldots, C_r\}$  be 
the set of minimal vertex cover of $G$. 
If $C_j=(C_i\setminus \{t_{k'}\})\cup\{t_k\} $, we set 
$f^{i,j}=t_k e'_i- t_{k'} e'_j$, where $e'_1, \ldots, e'_r$ are 
the unit vectors of $S^r$. Thus, 
$$
F=\{f^{i,j} \mid \{C_i, C_j\} \in E(\mathcal{G}) \mbox{ and } i<j\}
$$
is the set of column vectors of the linear syzygy matrix 
$\mathcal{L}\mathcal{S}(I_c(G))$ of $I_c(G)$.

$\Rightarrow$) We argue by contradiction assuming there is $C^3=\{C_1, C_2,
C_3, C_1\}$ a strong $3$-cycle of $\mathcal{G}$.
Since $\{C_1, C_2\},  \{C_1, C_3\} \in E(\mathcal{G})$, we can write 
$C_2=(C_1 \setminus \{t_j\})\cup \{ t_i \}$ and  
$C_3=(C_1 \setminus \{t_{j'}\})\cup \{ t_{i'} \}$, where 
$t_j\in C_1$, $t_{i}\notin C_1$, $t_{j'}\in C_1$, $t_{i'}\notin C_1$. As $C^3$ is strong, 
then  $\{t_i\}=C_2\setminus C_1=C_3\setminus C_1=\{t_{i'}\}$, 
$t_i=t_{i'}$ and $t_j\neq t_{j'}$. Note that $t_{j'}\in C_2=(C_1
\setminus \{t_j\})\cup \{ t_i \}$ because  
$t_j\neq t_{j'}$, and $t_j\notin C_2$ because $t_i\neq t_j$.  
We claim that 
$C_3=(C_2 \setminus \{t_{j'}\}) \cup \{ t_j\}$. 
As both sides 
have the same cardinality, it suffices to show
the inclusion ``$\supset$''. Take $t\in(C_2 \setminus \{t_{j'}\})
\cup \{ t_j\}$. If $t=t_{j}$, then $t_{j}\in C_3$ because $t_j\neq
t_{j'}$ and $t_j\in C_1$. If $t=t_i$, then $t\in C_3$ because
$C_3\setminus C_1=\{t_i\}$. If $t\in C_2\setminus\{t_{j'}\}$, $t\neq
t_i$, then $t\in C_1\setminus\{t_{j'}\}\subset C_3$ because $C_2=(C_1 \setminus
\{t_j\})\cup \{ t_i \}$. This proves the claim. Hence, 
$f^{1,2}=t_i e'_1- t_j e'_2$, $f^{1,3}=t_i e'_1- t_{j'}e'_3$ and 
$f^{2,3}=t_j e'_2- t_{j'}e'_3$. Thus,
$f^{2,3}=f^{1,3}-f^{1,2} \in S(f^{1,2}, f^{1,3})$.  
Therefore, $F$ is not a minimal system of generators for $S(F)$, a contradiction.

$\Leftarrow$) Arguing by contradiction we may assume 
$f^{1,2}=t_q e'_1-t_{q'} e'_2 \in S(F')$, where 
$F'=F \setminus \{f^{1,2}\}$. Thus, 
$C_2=(C_1 \setminus \{t_{q'}\}) \cup \{t_q\}$ and 
$f^{1,2}=\sum_{f^{i,j} \in F'} g_{i,j}f^{i,j}$, where
$g_{i,j} \in S$ for each $f^{i,j} \in F'$. 
Now, taking the inner product of $f^{1,2}$ with $e'_1$, we obtain 
\begin{equation}\label{jul31-24}
t_q=\langle f^{1,2}, e'_1 \rangle =\sum_{f^{1,j} \in F'} g_{1,j} \langle f^{1,j}, e'_1 \rangle=
\sum_{f^{1,j} \in F'} g_{1,j} t_{k_{1,j}},
\end{equation}
where $f^{1,j}=t_{k_{1,j}}e'_1-t_{k'_{1,j}}e'_j$. If $t_{k_{1,j}}
\neq t_q$  for each 
$f^{1,j}  \in F'$, then making $t_{k_{1,j}}=0$ for all $j \neq 1$ and
$t_q=1$ in Eq.~\eqref{jul31-24}, we obtain $1=0$, a contradiction. Then, there is 
$f^{1,j} \in F'$ such that $t_{k_{1,j}}=t_q$.  Without loss of 
generality, we may assume $j=3$, then $f^{1,3}=t_q e'_1-t_{k'_{1,3}}e'_{3}$. 
Consequently, 
$C_3=(C_1 \setminus \{t_{k'_{1,3}}\}) \cup \{ t_q \}$.
We can proceed as in the first part of the proof to show the equality
$C_3=(C_2 \setminus \{t_{k'_{1,3}}\}) \cup \{t_{q}\}$. Then,   
$C_2\setminus C_1=\{t_q\}=C_3\setminus C_1$, and 
$\{C_1, C_2, C_3, C_1\}$ is a strong $3$-cycle of $\mathcal{G}$, a
contradiction. 
Therefore, the set of columns of
$\mathcal{L}\mathcal{S}(I_c(G))$ is a minimal system of generators for
$S(F)$.
\end{proof}

\section{Linearly presented ideals of covers}\label{section-lpi}

Let $G$ be an unmixed graph. In this section we show that the ideal of
covers $I_c(G)$ of $G$ is linearly presented if $G$ has no induced $4$-cycles and
classify combinatorially when $I_c(G)$ is linearly presented for
graphs without $3$- and $5$-cycles.

\begin{proposition} \label{propertyP}
An edge $e$ has the  {\rm (P)} property  if and only if 
$\vert e \cap C  \vert =1$ for each minimal vertex cover $C$ of $G$.
\end{proposition}

\begin{proof} $\Rightarrow$)  
By contradiction suppose there is a minimal vertex cover $C$ of $G$
such that $ e \subset C$. We assume  $e=\{t_1,t_2\}$,  then there are
$t_1',t_2' \in V(G) \setminus C$ such that $\{t_1,t_1'\},
\{t_2,t_2'\} \in E(G)$. Indeed, as $C$ is minimal, there are
$\{x,y\}\in E(G)$ and $\{z,w\}\in E(G)$ such that 
$$
(C\setminus\{t_1\})\cap\{x,y\}=\emptyset\ \mbox{ and }\ 
(C\setminus\{t_2\})\cap\{z,w\}=\emptyset.
$$
We may assume $x\in C$ and $z\in C$. Then, $x=t_1$, $z=t_2$, 
$y\notin C$, and $w\notin C$. We let $t_1'=y$ and $t_2'=w$. As $e$ has the  {\rm (P)} property, then 
$\{t_1',t_2'\} \in E(G)$. Hence, $\{t_1',t_2'\} \cap C \neq \emptyset$,
since $C$ is a vertex cover, a contradiction.

$\Leftarrow$) We let $e=\{t_1,t_2\}$. To show that $e$ has the (P)
property let $\{t_1,t'_1\}$ and $\{t_2,t'_2\}$ be two edges of $G$
distinct from $e$. We argue by contradiction assuming 
$\{t'_1,t'_2\}$ is not an edge of $G$. Then, there exists a minimal vertex cover
$C$ of $G$ such that $C\cap\{t_1',t_2'\}=\emptyset$. Then, $t_i\in C$
for $i=1,2$ and $|C\cap e|=2$, a contradiction. 
\end{proof} 

\begin{proposition} \label{propertyP2}
If $C^4$ is a $4$-cycle of  a graph $G$, then the following conditions are equivalent:
\begin{enumerate}
\item[\rm(a)] Each edge of $C^4$ has the {\rm (P)} property.
\item[\rm(b)] $C^4$ has two disjoint edges with the {\rm (P)} property.
\item[\rm(c)] $\vert V(C^4) \cap C  \vert= 2$ for every minimal vertex cover $C$ of $G$.
\end{enumerate}
\end{proposition}

\begin{proof}
(a)$\Rightarrow$(b) This implication is clear.

(b)$\Rightarrow$(c) There are $\epsilon_1, \epsilon_2 \in E(C^4)$ such that they have the 
 {\rm (P)} property and $\epsilon_1 \cap \epsilon_2=\emptyset$. Then, $\epsilon_1 \cup \epsilon_2=V(C^4)$.
We take a minimal vertex cover $C$ of $G$. Thus, by Proposition~\ref{propertyP}, 
$\vert \epsilon_i  \cap C  \vert =1$ for $i=1,2$. Hence, $\vert V(C^4) \cap C  \vert= 2$,
since $\epsilon_1 \cap \epsilon_2=\emptyset$. 

(c)$\Rightarrow$(a) We argue by contradiction assuming that 
$C^4=\{t_1,t_2,t_3,t_4,t_1\}$ is
a $4$-cycle of $G$ such that 
$\{t_1,t_2\}$ does not have the {\rm (P)} property. Then, there are 
$\{t_1,t'_1\},\{t_2,t'_2\} \in E(G)$ such that $\{t'_1,t'_2\} \notin E(G)$. Consequently,
there is a maximal stable set $A$ of $G$ such that $\{t'_1,t'_2\}
\subset A$.
So, $t_1 \notin A$ and $t_2 \notin A$. Furthermore $t_3 \notin A$ or $t_4 \notin A$,
since $\{t_3,t_4\} \in E(G)$.  We may assume $t_3 \notin A$. Since $A$ is a maximal
stable set, we have that $C:=V(G) \setminus A$ is a minimal vertex cover and 
$t_1, t_2, t_3  \in C$, a contradiction, since $\vert V(C^4) \cap C  \vert= 2$.
 \end{proof}

\begin{lemma}\label{Property-duplicated}
If $C^4=\{t_1,t_2,t_3,t_4,t_1\}$ is a $4$-cycle of a graph $G$ 
and each edge of $C^4$ has the {\rm (P)} property, then  
$N_G(t_1)=N_G(t_3)$, $N_G(t_2)=N_G(t_4)$, and $C^4$ is an induced cycle.
\end{lemma}

\begin{proof} To show the inclusion $N_G(t_1)\subset N_G(t_3)$ take 
$t\in N_G(t_1)$, then $\{t,t_1\}$ and $\{t_2,t_3\}$ are edges of $G$.
We may assume $t\neq t_2$, otherwise $t\in N_G(t_3)$. 
As $\{t_1,t_2\}$ has 
the {\rm (P)} property, $t \neq t_3$ and $\{t,t_3\}\in E(G)$. 
Thus, $t\in N_G(t_3)$.  To show the inclusion ``$\supset$'' take 
$t'\in N_G(t_3)$, then $\{t',t_3\}$ and $\{t_1,t_4\}$ are edges of $G$.
We may assume $t'\neq t_4$, otherwise $t'\in N_G(t_1)$. 
As $\{t_3,t_4\}$ has 
the {\rm (P)} property, $t' \neq t_1$ and $\{t',t_1\}\in E(G)$. 
Thus, $t'\in N_G(t_1)$. The equality $N_G(t_2)=N_G(t_4)$ follows using
similar arguments. From the two equalities, we get that $C^4$ is an
induced cycle since $G$ is a simple graph.
\end{proof} 

\begin{theorem}\label{connected-duplicated}
Let $G$ be an unmixed graph without isolated vertices and let $I_c(G)$
be its ideal of covers. If the graph 
$\mathcal{G}_{I_c(G)}$ of $I_c(G)$ is connected, then 
\begin{enumerate}
\item[\rm(a)] $G$ has no duplicated vertices, and 
\item[\rm(b)] Every  $4$-cycle of $G$ has an edge that does not
satisfy the {\rm (P)} 
property. 
\end{enumerate}
\end{theorem}
 
 \begin{proof} (a)  Let $E(G^\vee)$ be the set of the minimal vertex
 covers of $G$. 
By contradiction suppose $t$ and $t'$ are two duplicated vertices of $G$.  
 We let
 $$
 A=\{C \in E(G^\vee) \mid  t,t' \in C\} \mbox{ and } B=\{C
 \in E(G^\vee) \mid  t \notin C \mbox{ and } t' \notin C \}. 
$$
\quad Now assume there is $C \in E(G^\vee) \setminus (A \cup B)$. Thus, 
 $\vert \{t,t'\} \cap C \vert =1$.  We may assume $t \in C$  and $t' \notin C$, 
 then  $N_G(t') \subset C$. Since  $t$ and $t'$ are duplicated,  $N_G(t)=N_G(t')$.
 Consequently, $C \setminus \{t\}$ is a vertex cover, a contradiction,
 since $C$ is minimal. Hence, $E(G^\vee)=A \cup B$. As $t$ is not an
 isolated vertex, by Lemma~\ref{existencia-cubiertas}, there are
minimal vertex covers  $C$ and $C'$ such that $t\in C'$ and 
$t\notin C$. Then, $C'\notin B$ (resp. $C\notin A$) because $t\in C'$
(resp. $t\notin C$). Thus, by the equality $E(G^\vee)=A \cup B$, one
has $C'\in A$ and $C\in B$. As $\mathcal{G}_{I_c(G)}$ is connected,
there is a path $\mathcal{P}=\{C_1=C',C_2,\ldots,C_n=C\}$ between $C'$
and $C$. Pick $i$ such that $C_i\in A$ and $C_{i+1}\in B$. Hence, 
$\{t,t'\}\subset C_i \setminus C_{i+1}$ and 
$|C_i \setminus C_{i+1}| \geq 2$, a contradiction since
$\{C_i,C_{i+1}\}\in E(\mathcal{G}_{I_c(G)})$. 

(b) By contradiction assume that $C^4$ is a 4-cycle of $G$ whose
edges have the (P) property, then $G$ has duplicated vertices by
Lemma~\ref{Property-duplicated}, a contradiction to part (a). 
 \end{proof} 

\begin{remark}\label{induced-4-cycle}
 If $C^4=\{t_1,t_2,t_3,t_4,t_1\}$ is a $4$-cycle with chord
$\{t_1,t_3\}$, then $\{t_1,t_2\}$  does not satisfy 
the {\rm (P)} property, since $\{t_1,t_3\}, \{t_2,t_3\} \in E(G)$, 
$\{t_3,t_3\} \notin E(G)$ and recalling that $G$ is a simple graph.
Therefore, condition (b) of Theorem~\ref{connected-duplicated}
is equivalent to (b') \textit{Every induced $4$-cycle has an edge 
without the {\rm (P)} property}. 
\end{remark}

\begin{proposition}\label{konig type}
If $P$ is a perfect matching with the {\rm (P)} property of a graph $G$, 
then $P$ is a perfect matching of $G$ of K\"onig type.
\end{proposition}

\begin{proof} 
We assume $P=\{\epsilon_1, \ldots, \epsilon_m\}$ and we take
$C$ a minimal vertex cover of $G$. Then,
by Proposition~\ref{propertyP}, $\vert \epsilon_i \cap C  \vert =1$ 
for each $1 \leq i \leq m$. Thus, $\vert C \vert= m$,
since $P$ is a partition of $V(G)$. Hence,
$G$ is unmixed and $\alpha_0(G)= m=|P|$, that is, 
$P$ is a perfect matching of  K\"onig type.
\end{proof}

\begin{remark}\label{Koning-type}
By Proposition~\ref{konig type}, we can 
remove the hypothesis that $P$ is of K\"onig type in parts (a) and
(b) of Proposition~\ref{unmixed-Koning-CM}. 
\end{remark}

\begin{theorem}\label{Konig-connectd}
 Let $G$ be an unmixed K\"onig graph without isolated vertices. The following
conditions are equivalent. 
\begin{enumerate}
\item[\rm(a)] G is Cohen--Macaulay.
\item[\rm(b)] $I_c(G)$ is linearly presented.
\item[\rm(c)] $\mathcal{G}_{I_c(G)}$ is connected.
\item[\rm(d)] $G$ has no duplicated vertices.
\item[\rm(e)] Every induced $4$-cycle of $G$ has an edge without the {\rm (P)} property.
\item[\rm(f)] $G$ has a unique perfect matching.
\end{enumerate}
\end{theorem}

\begin{proof}  
(a) $ \Rightarrow$ (b)  By Eagon-Reiner criterion $I_c(G)$ has a
linear resolution (Theorem~\ref{eagon-reiner-terai}). 
Hence, $I_c(G)$ is linearly presented.

(b) $ \Rightarrow$ (c) $\Rightarrow$ (d)  $\Rightarrow$ (e)
These implications follow from 
Proposition~\ref{Bigdeli-Herzog-j},
Theorem~\ref{connected-duplicated}, and 
Lemma~\ref{Property-duplicated}, respectively.

(e) $\Rightarrow$ (a) By Proposition~\ref{unmixed-Koning-CM}(a), 
$G$ has a perfect matching  $P$ of K\"onig type
with the  {\rm (P)} property.
By contradiction assume $G$ is not Cohen--Macaulay.
Then, by Proposition~\ref{unmixed-Koning-CM}(b), there is a
$4$-cycle $C^4$ with two edges in $P$. 
Thus, by Proposition~\ref{propertyP2} (cf.
Lemma~\ref{Property-duplicated}),
 each edge of $C^4$ 
has the  {\rm (P)} property. 
Hence, by Remark~\ref{induced-4-cycle}, $C^4$ is induced, a contradiction to (e).

(a)$\Leftrightarrow$(f) An unmixed K\"onig graph $G$ without isolated
vertices has a perfect matching of K\"onig type 
\cite[Lemma~2.3]{MRV}. Thus, $G$ is very well-covered. 
Hence, by \cite[Theorem~0.2]{Crupi-Rinaldo-Terai}, $G$ is Cohen--Macaulay if and only
if $G$ has a unique perfect matching. 
\end{proof}

\begin{lemma} \label{lemme-G(I_c(C))}
Let $G,H$ be unmixed graphs such that $H \subset G$. If  $V(H)=V(G)$ and 
$\alpha_0(G)=\alpha_0(H)$, then 
$\mathcal{G}_{I_c(G)} \subset \mathcal{G}_{I_c(H)}$.
\end{lemma}

\begin{proof}  
We set  $\mathcal{G}_1=\mathcal{G}_{I_c(G)}$ and
$\mathcal{G}_2=\mathcal{G}_{I_c(H)}$. We take
$C \in V(\mathcal{G}_1)$, then $C$ is a minimal vertex
cover of the graph $G$. 
So, $C \subset V(G)=V(H)$ and $|C|=\alpha_0(G)
=\alpha_0(H)$. Thus, $C$ is a minimal vertex cover of $H$, since
$E(H) \subset E(G)$. Hence,  $C \in V(\mathcal{G}_2)$.
Furthermore, if 
 $\{C_1,C_2\} \in E(\mathcal{G}_1)$, then 
 $\{C_1,C_2\} \in E(\mathcal{G}_2)$.
Therefore, $\mathcal{G}_1 \subset \mathcal{G}_2$.
\end{proof}

\begin{theorem}\label{CM-connectd1}
If $G$ is a Cohen--Macaulay K\"onig graph without isolated vertices, then 
\begin{enumerate}
\item[\rm(a)] $G$ has a unique perfect matching $P$ which is the set
of exchange edges of $G$, and
\item[\rm(b)] $\mathcal{G}_{I_c(G)}$ is a bipartite graph.
\end{enumerate}
\end{theorem}

\begin{proof}(a) By induction on $\vert V(G) \vert\geq 2$. If
$|V(G)|=2$, the only Cohen--Macaulay K\"onig graph without isolated
vertices is an edge $\{t_1,t_2\}$ and $P=\{\{t_1,t_2\}\}$ is the 
only perfect matching. To see that $\{t_1,t_2\}$ is an exchange edge 
note that $C_1=\{t_1\}$ and $C_2=\{t_2\}$ are minimal vertex covers of
$G$ and $(C_1\setminus\{t_1\})\cup\{t_2\}=C_2$. Assume $|V(G)|>2$. 
From Proposition~\ref{unmixed-Koning-CM}(b), $G$ has a perfect 
matching $P$ with the  {\rm (P)} property. 
By 
Proposition~\ref{CM-free-vertex}, there is $t \in V(G)$ such that $N_G(t)=\{t'\}$. 
So, $\epsilon_1=\{t,t'\} \in P$, since $P$ is a partition of $V(G)$.
Now, we let 
$G'=G \setminus N_G[t]$, then $G'$ is Cohen--Macaulay because being
Cohen--Macaulay is closed under c-minors \cite[6.3.53]{monalg-rev}. Furthermore, 
$P':=P \setminus \{\epsilon_1\}$ is a perfect matching of $G'$ with 
the  {\rm (P)} property in $G'$, since $G'=G \setminus \{t,t'\}$. Thus,
by Proposition~\ref{propertyP},  $\alpha_0(G')=\vert P' \vert$, since $P'$ is a partition
of $V(G')$.  Then, $G'$ is a  K\"onig graph. Hence, by induction hypothesis, 
$P'$ is the unique perfect matching  of $G'$ and $P'=Ex(G')$, where 
$Ex(G)$ and $Ex(G')$ are the exchange edges of $G$ and $G'$,
respectively. 

We let $\mathcal{C}=E(G^\vee)$ and $\mathcal{C}'=E((G')^\vee)$ be 
the set of minimal vertex covers of
the graphs $G$ and $G'$, respectively. Take $\{ t_1,t_2\} \in P'=Ex(G')$, then
there are  
$C'_1, C'_2 \in \mathcal{C}'$ such that $C'_2=(C'_1 \setminus \{t_1\}) \cup \{t_2\}$.
Consequently, $C_1=C'_1 \cup \{t'\}, C_2=C'_2 \cup \{t'\} \in  \mathcal{C}$, since
$G'=G \setminus \{t,t'\}$ and $N_G(t)=\{t'\}$. Furthermore, if  
$C_2=(C_1 \setminus \{t_1\}) \cup \{t_2\}$, then $\{t_1,t_2\} \in Ex(G)$. This implies,
$P' \subset Ex(G)$. On the other hand, by Lemma~\ref{existencia-cubiertas}, 
there is $\overline{C} \in \mathcal{C}$ such that
$t \in \overline{C}$, then $(\overline{C} \setminus \{t\}) \cup \{t'\} \in \mathcal{C}$,
since $N_G(t)=\{t'\}$. Hence, $\epsilon_1=\{t,t'\} \in Ex(G)$ and $P \subset
Ex(G)$. 

Now, we will prove that $P=Ex(G)$. By contradiction assume there is
$\{\overline{t}_1, \overline{t}_2\} \in Ex(G) \setminus P$, then there are 
$C_1,C_2 \in \mathcal{C}$ such that  $C_2=(C_1 \setminus \{\overline{t}_1\}) \cup \{\overline{t}_2\}$.
By Proposition~\ref{propertyP}, 
$$\vert C_1 \cap \epsilon_1 \vert=1=\vert C_2 \cap \epsilon_1 \vert,$$
since
$\epsilon_1$ has the  {\rm (P)} property. Thus, $C_1=C'_1 \cup \{\tilde{t}\}$ and
$C_2=C'_2 \cup \{\tilde{t}'\}$, where $C'_1,C'_2 \in E((G')^\vee)$ and 
$\{\tilde{t}, \tilde{t}'\} \subset \{t,t'\}$. So, $C'_1, C'_2 \in \mathcal{C}'$,
since $C_1,C_2 \in \mathcal{C}$ and 
$\alpha_0(G')=\vert P'  \vert=\vert P \vert -1=\alpha_0(G)-1$. If 
$\tilde{t} \neq \tilde{t}'$, then $C'_1=C'_2$ and
$\{\overline{t}_1,\overline{t}_2\}=\{t,t'\}=\epsilon_1$.
A contradiction, since $\{\overline{t}_1,\overline{t}_2\} \notin P$. 
Consequently, $\tilde{t}=\tilde{t}'$ and 
$C'_2=(C'_1 \setminus \{ \overline{t}_1 \}) \cup \{ \overline{t}_2 \}$. 
Hence, $\{\overline{t}_1,\overline{t}_2\} \in Ex(G')=P'\subset P$.
This is a contradiction. Therefore, $P=Ex(G)$.

(b) By Proposition~\ref{propertyP},  $\alpha_0(G)=\vert P \vert=\alpha_0(P)$. 
Thus, by Lemma~\ref{lemme-G(I_c(C))}, 
$\mathcal{G}_{I_c(G)} \subset \mathcal{G}_{I_c(P)}$,
since $P \subset G$ and $V(P)=V(G)$. Furthermore, 
by Proposition~\ref{I_c-bipartite},  the graph $\mathcal{G}_{I_c(P)}$ is bipartite. 
Hence, $\mathcal{G}_{I_c(G)} $  is bipartite.
\end{proof}

Recall that if $G$ is a graph and $L\subset V(G)$, then $G[L]$ denotes the subgraph of $G$
induced by $L$. The vertex set of $G[L]$ is $L$ and $e\in E(G[L])$ if and
only if $e\subset L$ and $e\in E(G)$.

\begin{lemma}\label{Covers-sym-dif}
Let $G$ be an unmixed graph with distinct minimal vertex covers
$C_1$ and $C_2$. If $C'$ is a vertex cover of $G':=G[C_1 \triangle C_2]$, then the following
conditions hold: 
\begin{enumerate}
\item[\rm(a)]  $C' \cup A$ is a vertex cover of $G$, where $A:=C_1 \cap C_2$.
\item[\rm(b)] $G'$ is a bipartite graph with a perfect matching $P$ and
bipartition $(B_1,B_2)$ such that $\vert B_1 \vert=\vert B_2 \vert=\vert P \vert=\alpha_0(G')$,
where $B_1:=C_1 \setminus C_2$ and $B_2:=C_2 \setminus C_1$. 
\end{enumerate}
\end{lemma}

\begin{proof} 
(a) Let $C'$ be any vertex cover of $G'$. 
We prove that $C' \cup A$ is a vertex cover of $G$, where $A=C_1 \cap C_2$.
To show this, take $\epsilon \in E(G)$. We may assume $\epsilon \cap A=\emptyset$.
Since $C_i$ is vertex cover,  there is $t_i \in \epsilon \cap C_i$ for $i=1,2$. 
Then, $t_1 \notin C_2$ and  $t_2 \notin C_1$, since 
$\epsilon \cap (C_1 \cap C_2)=\epsilon \cap A=\emptyset$.
So, $t_1 \in C_1 \setminus C_2$ and $t_2 \in C_2 \setminus C_1$.
Thus, $\epsilon=\{t_1,t_2\} \subset  C_1 \triangle C_2=V(G')$.
Hence, $\epsilon \cap C' \neq \emptyset$, since $C'$ is a vertex
cover of $G'$. Therefore, $C' \cup A$ is a vertex cover of $G$.

(b) We have $B_1=C_1 \setminus C_2$ and $B_2= C_2 \setminus C_1$. 
In particular  $B_1 \cap C_2=\emptyset$ and $B_2 \cap C_1=\emptyset$, then
$B_1$ and $B_2$ are stable sets of $G$, since $C_2$ and $C_1$ are vertex covers. 
Consequently, $G'$ is  bipartite with bipartition $(B_1,B_2)$, since $V(G')=C_1 \triangle C_2=B_1 \cup B_2$
and  $B_1 \cap B_2= \emptyset$. Thus, by K\"onig's theorem
(Theorem~\ref{konig-theorem}),  
$G'$ is a K\"onig graph since $G'$ is bipartite. So, there is  a matching $P$ of $G'$ such that 
$\vert P \vert= \alpha_0(G')$.  Furthermore,
by (a),  we have $C' \cup A$ is a vertex cover of $G$ , for each vertex cover $C'$ of $G'$. Then,
\begin{equation}\label{vert-ago-24}
 \vert C' \cup  A \vert  \geq  \alpha_0(G) = \vert C_i \vert =\vert B_i \cup A \vert \mbox{  for }
 i=1,2,
\end{equation}
since $C_i$ is a minimal vertex cover and $C_i=B_i \cup A$.
Notice that $C' \cap A \subset V(G') \cap A=\emptyset$ and $B_i \cap A=\emptyset$.
Consequently, by Eq.~\eqref{vert-ago-24},  we have $\vert C' \vert  \geq  \vert B_i \vert$  for each vertex cover
$C'$ of $G'$. Thus, $\alpha_0(G') \geq \vert B_i  \vert$. Also, $B_1$, $B_2$ are vertex covers of $G'$,  
since  $B_2=V(G') \setminus B_1$, and $B_2=V(G') \setminus B_2$. 
Hence,  $\vert B_1 \vert=\alpha_0(G')=\vert B_2 \vert$. 
Recall  $\alpha_0(G')=\vert P \vert$, then 
$$\vert V(G' ) \vert = \vert B_1 \vert + \vert B_2 \vert =2
\alpha_0(G') =2 \vert P \vert.$$
\quad Therefore, $P$ is a perfect matching of $G'$.
\end{proof}

\begin{theorem}\label{noLinear-noCM}
If $G$ is an unmixed graph and $J=I_c(G)$
is not linearly presented, then there are
distinct minimal vertex covers $C_1$ and $C_2$ 
such that $G'=G[C_1 \triangle C_2]$ is an unmixed 
bipartite graph without isolated vertices and 
$G'$ is not Cohen-Macaulay.
\end{theorem}

\begin{proof} 
By Proposition~\ref{Bigdeli-Herzog-j}(b), there are $C_1, C_2  \in E(G^\vee)$
such that there is no path in 
$\mathcal{G}_J^{(C_1,C_2)}$ connecting $C_1$ and $C_2$, where
$J=I_c(G)$. We choose $C_1$
and $C_2$ such that $\vert C_1  \cap C_2 \vert$ is maximum 
with respect to this property.
One has $\alpha_0(G)=\vert C_1\vert=\vert C_2\vert$,
since $G$ is an unmixed graph. 
By Lemma~\ref{Covers-sym-dif}(b), $G'=G[C_1 \triangle C_2]$  is bipartite with 
a perfect matching $P$ and bipartition $(B_1,B_2)$ such that
 $\vert B_1 \vert=\vert B_2 \vert=\vert P \vert=\alpha_0(G')$,
where $B_1=C_1 \setminus C_2$ and $B_2=C_2 \setminus C_1$. 

Now we prove that $P$ has the {\rm (P)}  property in $G'$.
By contradiction, suppose there is an edge $\{t,t'\}$ in $P$ and $t_1,t_2 \in V(G')$ 
such that  $\{t, t_1 \}, \{t', t_2\} \in E(G')$  and $\{t_1, t_2\} \notin E(G')$. 
Since $G'$ is bipartite, we have $\vert  \{t, t', t_1, t_2 \} \vert = 4$.
Also, there is a maximal stable set $U'$ of $G'$ such that $\{t_1,t_2\} \subset U'$. Thus, 
$C'=V(G') \setminus U'$ is a minimal vertex cover of $G'$, 
$t_1, t_2 \notin C'$, and $\{t,t'\}\subset C'$.  
We may assume $t  \in B_1$ and  $t' \in B_2$, since $(B_1, B_2)$ is a bipartition of $G'$. 
Then, $t_1 \in B_2$ and $t_2 \in B_1$. We consider the following subsets of $P$:
$$
P_1= \{\epsilon \in P \mid   \mbox{ } \epsilon  \subset C' \},  \mbox{ } P_2= \{\epsilon 
\in P \mid   \mbox{ } \epsilon  \cap C'  \subset B_1\} \mbox{ and }  
P_3= \{\epsilon  \in P \mid   \mbox{ } \epsilon  \cap C'  \subset B_2\}, 
$$
then $P_1, P_2, P_3$ are non-empty (see the discussion below) and form a partition of $P$, 
since $C'$ is a vertex cover of $G'$. 
We also consider the sets 
$$
B^i_j=V(P_i) \cap B_j, \mbox{ for } i=1,2,3 \mbox{ and } j=1,2.
$$  
\quad The $B^i_j$'s are mutually disjoint, since $P_1, P_2,
P_3$ is a partition of $P$ and $B_1 \cap B_2=\emptyset$.  Note that 
$\vert B_1^i\vert= \vert P_i\vert = \vert B_2^i\vert$ 
for $i=1,2,3$. Since $P$ is a perfect matching of $G'$, there are
$\{t_1,t'_1\}, \{t_2,t'_2\} \in P$. Consequently,  
$t'_1 \in B_1 \cap C'$ and $t'_2 \in B_2 \cap C'$, since $t_1 \in
B_2$,  $t_2 \in B_1$ and $t_1,t_2 \notin C'$. 
This implies $t'_1 \in B^2_1$, $t'_2 \in B^3_2$, $\{t_1,t_1'\}\in P_2$
and $\{t_2,t_2'\}\in P_3$.
Next we prove the equality 
\begin{equation}\label{jul6-24}
C'=B_1^1 \cup B_1^2 \cup B_2^1 \cup B_2^3.
\end{equation}
\quad The
inclusion ``$\supset$'' is clear by construction of the $P_i$'s. 
To show the inclusion ``$\subset$''
take $t_i\in C'$. There is $t_i'$ such that $\{t_i,t_i'\}\in P$. If
$t_i'\in C'$, then $\{t_i,t_i'\}\in P_1$ and $t_i\in B_1^1$ or $t_i\in
B_2^1$. If
$t_i'\notin C'$ and $\{t_i,t_i'\}\in P_2$, then $t_i\in B_1^2$. If
$t_i'\notin C'$ and $\{t_i,t_i'\}\in P_3$, then $t_i\in B_2^3$. This
proves the equality. Now, by Lemma~\ref{Covers-sym-dif}(a), we have
$C' \cup A$ is a vertex cover of $G$ where $A:=C_1 \cap C_2$.
Thus, there is  $C_3  \in E(G^\vee)$ such that $C_3 \subset C' \cup A$. 
We claim that $C' \subset C_3$. We argue by contradiction assuming 
there is $t_i \in C' \setminus C_3$. As $C'$ is a minimal vertex cover of $G'$, 
there is $\epsilon \in E(G')$ such that $\epsilon \cap C' = \{t_i\}$. Since
$\epsilon \in E(G')$, we have $\epsilon \cap A= \emptyset$. So, 
$\epsilon \cap C_3 \subset \epsilon \cap (C' \cup A) \subset \epsilon \cap C'=\{t_i\}$.
But $t_i \notin C_3$, then $\epsilon \cap C_3 =\emptyset$, a contradiction
since $C_3  \in E(G^\vee)$.
From the claim, we can write $C_3=A_1 \cup C'$ for some $A_1 \subset A$. 
Then, by Eq.~\eqref{jul6-24} and noticing that $B_1=B^1_1 \cup  B^2_1
\cup  B^3_1$, we obtain that $C_3$ and $C_1$ can be partitioned as 
\begin{align}
&C_3=A_1\cup C'=A_1 \cup B_1^1 \cup B_1^2 \cup B_2^1 \cup
B_2^3,\label{jul6-24-1}\\
&C_1=A \cup B_1=A_1 \cup A_2 \cup B^1_1 \cup  B^2_1 \cup  B^3_1,
\mbox{ where } A_2=A \setminus A_1,\label{jul6-24-2} 
\end{align}
respectively. Hence, noticing 
that $\vert C_1 \vert=\vert C_3 \vert$ and $ \vert B^3_1 \vert  = \vert P_3 \vert=
\vert B^3_2 \vert$, we obtain $\vert A_2 \vert  = \vert B^1_2 \vert$. 
From Eqs.~\eqref{jul6-24-1} and \eqref{jul6-24-2}, we have $C_1  \cap
C_3 = A_1 \cup B^1_1  \cup B^2_1$. Similarly, we obtain the equality 
$C_2 \cap C_3 = A_1 \cup B^1_2  \cup B^3_2$, since  $C_2=A \cup
 B_2=A_1 \cup A_2 \cup B^1_2 \cup  B^2_2 \cup  B^3_2$.   
Recall $\vert B_1^1 \vert=\vert P_1  \vert= \vert B_2^1 \vert$
and $ \vert B^1_2 \vert= \vert A_2 \vert$, then
$\vert C_1  \cap C_3  \vert = \vert A_1 \vert +\vert B^1_1 \vert
+ \vert B^2_1\vert=\vert A \vert+ \vert B^2_1 \vert$, 
since $A=A_1 \cup A_2$. Similarly, we obtain 
that $\vert C_2  \cap C_3 \vert = \vert A \vert+ \vert B^3_2 \vert$,
since  $\vert B^1_2 \vert= \vert A_2 \vert$ and $A=A_1\cup A_2$. 
Then,  
$$
\vert C_i  \cap C_3  \vert  > \vert A \vert \mbox{ for } i=1,2,
\mbox{ since } t'_1 \in B_1^2 \mbox{ and } t'_2  \in B_2^3. 
$$
\quad But $C_1$, $C_2$ were chosen so that $\vert C_1 \cap C_2 \vert=
\vert A \vert$ 
is maximal with respect to the property that
there is no path in $\mathcal{G}_J^{(C_1,C_2)}$ connecting $C_1$ and $C_2$.
Consequently, there is a path  $\mathcal{P}_i$ between $C_i$ and $C_3$ in
$\mathcal{G}_J^{(C_i,C_3)}$ for $i=1, 2$. Recall that $A_1\subset
A=C_1\cap C_2$ and $C'\subset B_1\cup B_2\subset C_1\cup C_2$. Thus, 
$C_3 \subset C_1 \cup C_2$, and one has  
$$
\mathcal{P}_i \subset \mathcal{G}_J^{(C_i,C_3)}  \subset \mathcal{G}_J^{(C_1,C_2)}  \mbox{ for } i=1,2.
$$
\quad This implies, $ \mathcal{P}_1  \cup  \mathcal{P}_2$ is a path between $C_1$ and $C_2$ in
$ \mathcal{G}_J^{(C_1, C_2)} $, a contradiction. Therefore, $P$ has the  {\rm (P)} property.
On the other hand,  by K\"onig's theorem (Theorem~\ref{konig-theorem}), 
$G'$ is a K\"onig graph, since $G'$ is bipartite. 
Also, $G'$ does not have isolated vertices, since $P$ is a perfect
matching of $G'$.
Hence, by Proposition~\ref{unmixed-Koning-CM}(a), $G'$ is an unmixed graph.

Now, suppose the graph $G'$ is Cohen-Macaulay.
By Proposition~\ref{CM-free-vertex},
there is $\overline{t} \in V(G')$ with $\deg_{G'}(\overline{t})=1$. We can assume $\overline{t} \in B_1$. 
Since $P$ is a perfect matching, there is $\{\overline{t},\overline{t}'\} \in P$
with $\overline{t}' \in B_2$. Consequently, $ \overline{t} \neq \overline{t}'$ and
$C_4=(C_1  \setminus \{\overline{t}\}) \cup \{\overline{t}'\}  \in E(G^\vee)$, 
since $N_{G'}(\overline{t})=\{\overline{t}'\} \subset C_4$ and 
$\vert C_4 \vert=\vert C_1 \vert$. Also, 
$C_4 \subset C_1 \cup C_2$ since $\overline{t}' \in B_2 \subset C_2$. 
Thus, $ \overline{\epsilon}=\{C_1,C_4\} \in E(\mathcal{G}_J^{(C_1, C_2)})$.
Furthermore, $C_4 \cap C_2 = A  \cup \{\overline{t}'\}$ and
$\overline{t}'\notin A$, since $\overline{t} \in B_1=C_1 \setminus C_2$ and
$\overline{t}'\in B_2=C_2\setminus A$. Then,
$\vert C_4 \cap C_2 \vert  >  \vert A  \vert = \vert C_1 \cap C_2 \vert $.
This implies, because of the way we chose $C_1$ and $C_2$, that  
there is a path $\mathcal{P}$ between $C_2$ and $C_4$ in
$\mathcal{G}_J^{(C_2, C_4)}$. So, $\mathcal{P} \in \mathcal{G}_J^{(C_1, C_2)}$, 
since $C_4 \subset C_1 \cup C_2$.
Hence, $\mathcal{P} \cup \{ \overline{\epsilon}\}$  is a path between $C_1$ and $C_2$
in $\mathcal{G}_J^{(C_1, C_2)}$, a contradiction.
Therefore, $G'$ is not Cohen-Macaulay.
\end{proof}

\begin{theorem}\label{4-cycles-connected}
If $G$ is an unmixed graph without induced $4$-cycles, then
$I_c(G)$ is linearly presented and ${\rm v}(I_c(G))=\alpha_0(G)-1$.
\end{theorem}

\begin{proof} 
By contradiction suppose that $I_c(G)$ is not linearly presented. 
Then, by Theorem~\ref{noLinear-noCM}, there are $C_1, C_2  \in E(G^\vee)$
such that  $G'=G[C_1 \triangle C_2]$ is not Cohen-Macaulay and $G'$
is an unmixed bipartite graph without isolated vertices. 
So, by K\"onig's theorem (Theorem~\ref{konig-theorem}), 
$G'$ is a K\"onig graph. 
Hence, by Theorem~\ref{Konig-connectd}(a)-(e), 
$G'$ is Cohen-Macaulay,
since $G$ does not have induced $4$-cycles. 
This is a contradiction. 
Therefore, $I_c(G)$ is linearly presented. Then, 
by Theorem~\ref{eclipse-apr8-24}(c), the v-number of $I_c(G)$ is
equal to $\alpha_0(G)-1$. 
\end{proof}

\begin{theorem}\label{not-linear-presented}
Let $G$ be an unmixed graph such that $I_c(G)$ is linearly presented.
If  $G[C_1 \triangle C_2]$ is unmixed and
not Cohen-Macaulay, where $C_1$ and $C_2$
are distinct minimal vertex covers of $G$, then there is $C \in E(G^\vee)$
such that $C_1 \cap C_2  \not\subset C \subset C_1 \cup C_2$.
\end{theorem}
\begin{proof} 
Assume $G':=G[C_1 \triangle C_2]$ is unmixed and not Cohen-Macaulay, 
where $C_1, C_2  \in E(G^\vee)$ and $C_1  \neq C_2$.
We will show there is $C \in E(G^\vee)$ such that
$A \not\subset C \subset C_1 \cup C_2$, where $A:= C_1 \cap C_2$. 
By hypothesis  $J:=I_c(G)$ is linearly presented. Thus, 
by (c) in Proposition~\ref{Bigdeli-Herzog-j},  
there is a path $ \mathcal{P}=\{\bar{C_1},\bar{C_2}, \ldots,
\bar{C_k}\}$
in $\mathcal{G}_J^{(C_1, C_2)}$ with $\bar{C_1}=C_1$, $\bar{C_k}=C_2$,
and $\bar{C_i} \subset C_1 \cup C_2$ for $i=1, \ldots k$. 
Furthermore, by Lemma~\ref{Covers-sym-dif}(b), $G'$ is a bipartite graph 
with a perfect matching  and bipartition $(B_1,B_2)$ such that $\vert B_1 \vert = \vert B_2 \vert$,
where $B_1:=C_1 \setminus C_2$ and $B_2:=C_2 \setminus C_1$. 
Since $G'$ has a perfect matching, $G'$
does not have isolated vertices. Also, by  K\"onig's theorem (Theorem~\ref{konig-theorem}), 
$G'$ is a K\"onig graph. Hence,
by Proposition~\ref{unmixed-Koning-CM}(a), $G'$ has a perfect matching 
$P$ of  K\"onig type with the {\rm (P)}  property,
since $G'$ is unmixed.
Then,  $\vert V(G') \vert = 2 \vert P \vert$ implies   
$\vert B_1 \vert = \vert B_2 \vert= \vert P \vert$, since    
$V(G')=B_1 \cup B_2$ and $B_1 \cap B_2=\emptyset$. 
This implies, $\vert \bar{C_i} \vert=\vert C_1 \vert=\vert A \vert +\vert P \vert$,
since $G$ is unmixed, $C_1=A \cup B_1$ and $A \cap B_1= \emptyset$.  

We claim that $A\not\subset\bar{C_i}$ for some $i$. Arguing by
contradiction, we assume $A \subset  \bar{C_i}$ for each $i \in \{1,
\ldots, k\}$. Letting $\bar{C'_i}:= \bar{C_i} \setminus A$, then 
$\vert  \bar{C'_i} \vert = \vert  \bar{C_i} \vert -  \vert  A  \vert = \vert P  \vert$.
We take $\epsilon \in P$. Since $ \bar{C_i}$ is a vertex cover, we have that
$\emptyset \neq  \bar{C_i} \cap \epsilon = \bar{C'_i} \cap \epsilon$,
because $A \cap \epsilon  \subset A \cap V(G') =\emptyset$.
Furthermore, $P$ is a matching, then 
$$
 \vert  \bar{C'_i} \vert \geq  \vert  \bar{C'_i} \cap V(P) \vert =  \sum_{\epsilon \in P} \vert \bar{C'_i} \cap \epsilon   \vert \geq  \vert P  \vert.
$$
But $\vert  \bar{C'_i} \vert = \vert P  \vert$, then $1=\vert \bar{C'_i} \cap \epsilon  \vert = \vert \bar{C_i} \cap \epsilon   \vert$
for each $\epsilon \in P$.
On the other hand $G'$ is not Cohen-Macaulay, then by
Proposition~\ref{unmixed-Koning-CM}(b), there is a $4$-cycle $(t_1, t'_1, t_2, t'_2, t_1)$
such that $\epsilon_1=\{t_1,t'_1\}, \epsilon_2=\{t_2,t'_2\} \in P$.  
Without loss of generality, we can assume $t_1,t_2 \in B_1$ and $t'_1,t'_2 \in B_2$, since $(B_1,B_2)$ is a bipartition of $G'$.
By a previous argument, $\vert \bar{C'_i} \cap \epsilon_1  \vert=1= \vert \bar{C'_i} \cap   \epsilon_2  \vert$, 
since  $\epsilon_1,  \epsilon_2 \in P$. So,  if $ \bar{C'_i} \cap \epsilon_1 =\{t_1\}$, then 
$\bar{C'_i} \cap \epsilon_2 =\{t_2\}$ since $\{t'_1,t_2\} \in E(G')$. Similarly, if $ \bar{C'_i} \cap \epsilon_1 =\{t'_1\}$, then 
$\bar{C'_i} \cap \epsilon_2 =\{t'_2\}$ since $\{t_1,t'_2\} \in E(G')$.
This  implies 
$$
\bar{C_i} \cap (\epsilon_1 \cup \epsilon_2)=\{t_1,t_2\} \mbox{ or }
\bar{C_i} \cap (\epsilon_1 \cup \epsilon_2)=\{t'_1,t'_2\},
$$
since $ \bar{C_i} \cap \epsilon = \bar{C'_i} \cap \epsilon$. In particular we have $C_1 \cap (\epsilon_1 \cup \epsilon_2)=\{t_1,t_2\}$ and 
$C_2 \cap (\epsilon_1 \cup \epsilon_2)=\{t'_1,t'_2\}$, since  $t_1,t_2 \in B_1$ and $t'_1,t'_2 \in B_2$.
Now, we let $j=\min\{i \mbox{ } \vert  \mbox{ }  \bar{C_i} \cap (\epsilon_1 \cup \epsilon_2)=\{t'_1,t'_2\}\}$, then
$j>1$ and $\{t'_1,t'_2\}\ \subset \bar{C_j} \setminus \bar{C}_{j-1}$.
This is a contradiction since 
$ \{ \bar{C_j},  \bar{C}_{j-1} \} \in E(\mathcal{G}_J^{(C_1, C_2)})$.
Therefore, there is $ \bar{C}_{i'}$ such that $A \not\subset \bar{C}_{i'}$. 
Hence, to finish the proof, we can set $C:=\bar{C}_{i'}$, since $A
\not\subset \bar{C}_{i'}  \subset C_1 \cup C_2$.
\end{proof}

\begin{lemma} \label{3-5-cycles-linpresented1}
 Let $G$ be a graph without $3$- and $5$-cycles.  If $G'$ is an induced subgraph of $G$
 with a perfect matching $P$ and $\overline{t} \in V(G) \setminus
 V(G')$, then, the set
 $$B=\{ t \in V(G) \mbox{ }  \vert \mbox{ there is } t' \in
 N_G(\overline{t}) \mbox{ such that }  \{t, t' \} \in P\}\ \mbox{ is
 a stable set.}
$$
\end{lemma}

\begin{proof} 
By contradiction assume there is $\{t_1,t_2\} \in E(G)$ such that $t_1, t_2 \in B$.
Then, $t_1 \neq t_2$ and there are $t'_1,t'_2 \in N_G(\overline{t})$ such that 
$\{t_1,t'_1\}, \{t_2,t'_2\} \in P$. If $t_1=t'_2$, then $\{\overline{t},t_1\} \in E(G)$ and
$(\overline{t}, t'_1, t_1, \overline{t} )$ is a $3$-cycle, a contradiction, since $G$ does not have
$3$-cycles. Hence $t_1 \neq t'_2$. Consequently, $\{t_1, t'_1\} \neq \{t_2, t'_2\} $,
since $t_1 \neq t_2$ and $t_1 \neq t'_2$. Thus, $\{t_1, t'_1\} \cap \{t_2, t'_2\}=\emptyset$,
since $P$ is a perfect matching. Then, $(\overline{t}, t'_1, t_1,
t_2, t'_2, \overline{t})$ is a $5$-cycle, a contradiction, since $G$
does not have $5$-cycles. 
Therefore, $B$ is a stable
set.
\end{proof} 

\begin{theorem} \label{3-5-cycles-linpresented}
 Let $G$ be an unmixed graph without $3$- and $5$-cycles. The
 following conditions are
 equivalent.
\begin{enumerate}
\item[\rm(a)] $I_c(G)$ is linearly presented.
\item[\rm(b)] If $G'=G[C_1 \triangle C_2]$ is an unmixed graph, where
$C_1, C_2 \in E(G^\vee)$, $C_1 \neq C_2$, then $G'$ is a
Cohen--Macaulay graph.
\end{enumerate}
\end{theorem}
\begin{proof} 
(b) $ \Rightarrow$ (a) This implication follows from Theorem~\ref{noLinear-noCM}.

(a) $ \Rightarrow$ (b). We argue by contradiction, assume there are  $C_1, C_2 \in E(G^\vee)$
such that $C_1 \neq C_2$, $G'=G[C_1 \triangle C_2]$ is unmixed and $G'$ is not Cohen--Macaulay. 
Then, by Lemma~\ref{Covers-sym-dif}(b), $G'$ is a bipartite graph with a perfect matching $P$. 
By Theorem~\ref{not-linear-presented} , there is $\bar{C}  \in E(G^\vee)$ such that
$A \not\subset \bar{C} \subset C_1 \cup C_2$, where $A:=C_1 \cap C_2$. 
So, there is $t_1 \in A \setminus \bar{C}$, then $t_1 \notin C_1 \triangle C_2=V(G')$ and $N_G(t_1) \subset \bar{C}$ 
since $ \bar{C}$ is a vertex cover. In particular
$N_G(t_1) \cap D = \emptyset$, where $D:=V(G)\setminus (C_1 \cup C_2)$,
since $\bar{C} \subset C_1 \cup C_2$.
Furthermore, by Lemma~\ref{3-5-cycles-linpresented1}
$$
B:=\{ t \in V(G') \vert \mbox{ there is } t' \in N_G(t_1) \mbox{ such that } \{t,t'\} \in P \}
$$
is a stable set. Consequently, there is a maximal stable set $L'$ of $G'$ such that $B \subset L'$. 
If $t \in L' \subset C_1 \triangle C_2$, then $t \notin C_1$ or $t \notin C_2$.
Thus, $N_G(t) \subset C_1$ or  $N_G(t) \subset C_2$,
since $C_1$ and $C_2$ are vertex covers of $G$. So, $N_G(t) \cap D= \emptyset$.
This implies, $L'  \cup D$  is a stable set. Also if $\{ t_1, \bar{t}\} \in E(G)$ with $\bar{t} \in L'$,
then there is $\{\bar{t},  \bar{t'} \}  \in P$. Consequently,
$\bar{t'} \in B \subset L'$, a contradiction, since $\bar{t}, \bar{t'} \in L'$ and $L'$ is a stable set.
Hence, $L=L' \cup D  \cup \{t_1\} $ is a stable set, since $N_G(t_1) \cap D= \emptyset$.
Furthermore, $L' \cap D = \emptyset$ and $t_1 \notin L' \cup D$, since 
$L' \subset C_1 \triangle C_2$ and $t_1 \in A$. Thus, $\vert L \vert=\vert L' \vert +\vert C \vert+1$.
On the other hand, $L_1:=V(G) \setminus C_1$ is a maximal stable set, since 
$C_1 \in E(G^\vee)$. Notice $L_1=B_2 \cup D$ where $B_2:=C_2 \setminus C_1$,
since $V(G)= (C_1 \cup C_2) \cup D$. In particular $B_2$ is a stable set and $B_2 \subset V(G')$.
Then, $\vert L' \vert \geq  \vert B_2 \vert$ implies 
$$
\vert L \vert=\vert L' \vert +\vert D \vert+1 \geq \vert B_2 \vert  +\vert D \vert +1  \geq
\vert L_1 \vert +1,
$$
This is a contradiction, since $G$ is unmixed. 
\end{proof}

\begin{proposition}\label{jul27-24} Let $U,U_1,U_2,U_3,U_4,U_5$ be the families of
graphs defined in Section~\ref{intro-section}. 
Then, 
$$
U_5 \subset U_4 \subset U_3 \subset U_2 \subset U_1 \subset U\ \mbox{ and all inclusions are strict}.
$$
\end{proposition}

\begin{proof} 
By Theorem~\ref{4-cycles-connected},
Proposition~\ref{Bigdeli-Herzog-j}, Theorem~\ref{connected-duplicated} 
and Lemma~\ref{Property-duplicated} all inclusions hold. To show that
all inclusions are strict note that the graph of Example ~\ref{contraexample}
is  in $U_4 \setminus U_5$, the graph of Example ~\ref{example-conjetura2}
is  in $U_3 \setminus U_4$, and the graph of Example ~\ref{example-conjetura}
is in $U_2 \setminus U_3$. Now we take a homogeneous $k$-partite graph $G$ 
with $k \geq 3$ such that $G$ is not a complete graph.
Then,  by (b) and (c) in Proposition~\ref{compl-$k$-partite}, 
$G \in U_1$. Furthermore, as $G$ is not a complete graph, one has 
$G=G_1*\cdots*G_k$, with $\vert G_i \vert \geq 2$. Thus,
$G \notin U_2$, since the vertices of $G_1$ are duplicated. 
Hence, $G \in U_1 \setminus U_2$. Finally if $G'$ is a homogeneous
complete bipartite graph with $\vert V(G)\vert \geq 4$, then
$G'$ has a $4$-cycle. Therefore, by  
(a) and (c) in Proposition~\ref{compl-$k$-partite}, 
 $G' \in U \setminus U_1$.
\end{proof}

\section{Examples}\label{examples-section}

\begin{example}\label{5-cycle} Let $G=C^5$ be a $5$-cycle with vertex set
$V(G)=\{t_1,\ldots,t_5\}$ and let $C$ be the minimal vertex cover 
$\{t_1,t_3,t_4\}$ of $G$. Then, 
\begin{align*}
&\ t_2\notin C,\ t_3\in C,\ 
\{t_2,t_3\}\in E(G),\ N_G(t_3)\setminus C=\{t_2\},\ \mbox{and}\\
&N_G(t_3)=\{t_2,t_4\}\not\subset N_G[t_2]=\{t_1,t_2,t_3\}. 
\end{align*}
\quad More generally one has $N_G(t_i)\not\subset N_G[t_j]$
for any $\{t_i,t_j\}\in E(G)$. Note that all minimal vertex covers of $G$ 
satisfy the exchange property. 
Thus, the converse of Proposition~\ref{saha->exchange} does not hold.
If $C^s$ is a cycle of length $s$, then ${\rm v}(I_c(C^s))=\lfloor
\frac{s}{2}\rfloor$ \cite[Proposition~3.3]{Saha}. In our case ${\rm
v}(I_c(G))=2$. The graph $\mathcal{G}_{I_c(G)}$ of $I_c(G)$ is again a $5$-cycle and, by
Proposition~\ref{min-gen-S}, the columns of the linear syzygy matrix of
$I_c(G)$ are linearly independent. 
\end{example}

\begin{example}\label{example2}
Let $G=C^7$ be a $7$-cycle with vertex set
$V(G)=\{t_1,\ldots,t_7\}$. The ideal $I_c(G)$ is minimally generated
by monomials of degree $4$. 
By Theorem~\ref{4-cycles-connected}, $I_c(G)$ is linearly presented.
By using \textit{Macaulay}$2$ \cite{mac2}, we can see that $I_c(G)$ does
not have a linear resolution. Furthermore, by Theorem~\ref{eclipse-apr8-24},
${\rm v}(I_c(G))=\alpha_0(G)-1=3$.
\end{example}  

\begin{example}\label{example-cograph}
Let $G_1$ and $G_2$ be two vertex disjoint 5-cycles. The join $G=G_1*G_2$
consists of $G_1\cup G_2$ and all the edges joining $V(G_1)$ and
$V(G_2)$. The complement $\overline{G}$ of $G$ is $G_1\cup G_2$ and $G$ is a 
codi-graph but $G$ is not a cograph because $G_1$ is an induced 
connected subgraph of $G$ whose complement is connected. By
Example~\ref{5-cycle} and Proposition~\ref{join-formula}, one has ${\rm
v}(I_c(G))=2+5=7$.
\end{example}

\begin{example}\label{example-join-formula}
Let $C^3$ and $C^5$ be two vertex disjoint odd cycles of length  $3$
and $5$, respectively, and let $G=C^3*C^5$ be their join. Then, 
${\rm v}(I_c(C^3))+|V(C^5)|=1+5=6$ and ${\rm
v}(I_c(C^5))+|V(C^3)|=2+3=5$. Thus, by Proposition~\ref{join-formula},
one has ${\rm v}(I_c(G))=5$. As $G$ is a codi-graph, one has 
${\rm reg}(S/I_c(G))=|V(G)|-2=6$ (Proposition~\ref{codi-graph-char}).
\end{example}

\begin{example}\label{contraexample}
Let $G$ be a graph consisting of an induced $4$-cycle 
$C^4=\{t_1,t_2,t_3,t_4,t_1\}$, and edges $\epsilon_1=\{t_1,t'_1\}$, $\epsilon_2=\{t_2,t'_2\}$, 
$\epsilon_3=\{t_3,t'_3\}$ and $\epsilon_4=\{t_4,t'_4\}$. Then, $P=\{\epsilon_1, \epsilon_2, \epsilon_3, \epsilon_4\}$
is a perfect matching with the {\rm (P)} property. Consequently, 
by Proposition~\ref{propertyP}, $\alpha_0(G)=4$.  
Thus, $P$ is a perfect matching of  K\"onig type
and $G$ is  K\"onig. Therefore, by Proposition~\ref{unmixed-Koning-CM}(b)   
and Theorem~\ref{Konig-connectd}, $G$ is Cohen--Macaulay 
and $I_c(G)$ is linearly presented. 
\end{example}

\begin{example}\label{example-conjetura}
 Let $G$ be the graph depicted in
Figure~\ref{figure-enrique} consisting of: a complete subgraph
$\mathcal{K}_4$ with 
$V(\mathcal{K}_4)=\{t_1,t_2,t_3,t_4\}$ and four $4$-cycles $C^{(1)}=\{t_1,t_2,t_6,t_5,t_1\}$,
 $C^{(2)}=\{t_1,t_8,t_7,t_3,t_1\}$, $C^{(3)}=\{t_2,t_9,t_{10},t_4,t_2\}$ and 
 $C^{(4)}=\{t_3,t_4,t_8,t_9,t_3\}$. Using
\textit{Macaulay}$2$ \cite{mac2}, one obtains that $G$ is an unmixed
graph with $\alpha_0(G)=6$ whose minimal vertex covers are: 
\begin{align*}
&C_1=\{t_2,t_3,t_4,t_5,t_8,t_9\},\
 C_2=\{t_2,t_3,t_4,t_5,t_8,t_{10}\},\ C_3=\{t_1,t_2,t_3,t_5,t_8,t_{10}\}, 
\\
&C_4=\{t_1,t_2,t_3,t_6,t_8,t_{10}\},\
C_5=\{t_1,t_3,t_4,t_6,t_8,t_9\},\ C_6=\{t_1,t_3,t_4,t_6,t_7,t_9\},
\\
&C_7=\{t_1,t_2,t_4,t_6,t_7,t_9\} \mbox { and }
C_8=\{t_1,t_2,t_4, t_5,t_7,t_9\}.
\end{align*}
\quad Hence,  $\mathcal{G}_{I_c(G)}$ consists of two disjoint path
${\mathcal{P}}_1=\{C_1, C_2, C_3, C_4\}$ and ${\mathcal{P}}_2=\{C_5, 
C_6, C_7, C_8\}$.
Therefore,  $\mathcal{G}_{I_c(G)}$ is not connected. Also, the exchange edges of $G$
are:
$$
\{t_9,t_{10}\},\  \{t_1, t_4\},\ \{t_5, t_6\},\ \{t_7, t_8\} \mbox{ and } \{t_2, t_3\}.
$$
\quad On the other hand note that the graph $G$ has no duplicated 
vertices. Then, by
Lemma~\ref{Property-duplicated}, every $4$-cycle has an edge 
without the {\rm (P)} property. Furthermore, the regularity of $S/I_c(G)$ 
is $6$ and ${\rm v}(I_c(G))=\alpha_0(G)-1=5$.
 
\begin{figure}[ht]
\begin{tikzpicture}[scale=2,thick]
	\tikzstyle{every node}=[minimum width=0pt, inner sep=1.8pt, circle]
	\draw (0.23,1.49) node[draw] (0) { \tiny $t_3$};
	\draw (1.12,1.52) node[draw] (1) { \tiny $t_1$};
	\draw (2.39,1.99) node[draw] (2) { \tiny $t_5$};
	\draw (0.27,0.57) node[draw] (3) { \tiny $t_4$};
	\draw (1.16,0.57) node[draw] (4) { \tiny $t_2$};
	\draw (2.43,0.08) node[draw] (5) { \tiny $t_6$};
	\draw (-1.46,2.1) node[draw] (6) { \tiny $t_8$};
	\draw (-0.45,1.68) node[draw] (7) { \tiny $t_7$};
	\draw (-1.44,-0.1) node[draw] (8) { \tiny $t_9$};
	\draw (-0.48,0.34) node[draw] (9) { \tiny $t_{10}$};
			\draw  (1) edge (6);
			\draw  (0) edge (4);
			\draw  (1) edge (3);
			\draw  (0) edge (3);
			\draw  (1) edge (4);
			\draw  (0) edge (1);
			\draw  (3) edge (4);
			\draw  (1) edge (2);
			\draw  (4) edge (5);
			\draw  (2) edge (5);
			\draw  (6) edge (7);
			\draw  (0) edge (7);
			\draw  (4) edge (8);
			\draw  (8) edge (9);
			\draw  (3) edge (9);
			\draw  (6) edge (8);
			\draw  (3) edge (6);
			\draw  (0) edge (8);
		\end{tikzpicture}
\caption{Unmixed graph without duplicated vertices, not K\"onig and $\mathcal{G}_{I_c(G)}$ 
not connected.}\label{figure-enrique}
\end{figure}
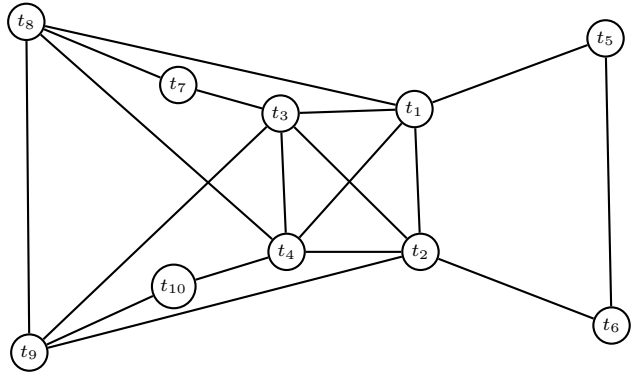
\end{example}

\begin{example}\label{example-conjetura2}
Let $G$ be the graph depicted in Figure~\ref{figure-enrique1} consisting of two cycles
$C^3=\{t_1,t_2,t_3,t_1\}$ and  $C^4=\{t_1,t_2,t_7,t_6,t_1\}$
and two edges $\{t_4,t_5\}$ and $\{t_3,t_4\}$. Using
\textit{Macaulay}$2$ \cite{mac2}, one obtains that $G$ is an unmixed
graph with $\alpha_0(G)=4$ whose minimal vertex covers are: 
\begin{align*}
&C_1=\{t_2,t_3,t_6,t_5\},\ C_2=\{t_2,t_3,t_6,t_4\},\
C_3=\{t_1,t_2,t_4,t_6\},\ C_4=\{t_1,t_2,t_4,t_7\},\\
&C_5=\{t_1,t_3,t_7,t_4\} \mbox{ and } C_6=\{t_1,t_3,t_7,t_5\}.
\end{align*}
\quad The graph $G$ has repeated exchange edges because the edge  
$\{t_4,t_5\}$ of $G$ is an exchange edge appearing in 
$(C_1\setminus\{t_5\})\cup\{t_4\}=C_2$ and
$(C_5\setminus\{t_4\})\cup\{t_5\}=C_6$. 
Letting $J=I_c(G)$, then  $\mathcal{G}_J$  is the path 
$\{C_1, C_2, C_3, C_4, C_5, C_6\}$ with $6$
vertices. 
Hence, $\mathcal{G}_J$ is connected. 
Furthermore, the subgraph $\mathcal{G}_J^{(C_1,C_6)}$
of $\mathcal{G}_J$ consist of two isolated vertices $C_1$ and $C_6$,
since $t_4 \notin C_1 \cup C_6$ and $t_4\in \cap_{i=2}^5 C_i$.
In particular $\mathcal{G}^{(C_1,C_6)}$ is not connected. 
Therefore, by Proposition~\ref{Bigdeli-Herzog-j}, $I_c(G)$ is not linearly presented. 
On the other hand, $\mathcal{G}_J^{(C_2,C_5)}$ is the path
$\{C_2, C_3, C_4, C_5\}$ and the exchange edges of this path
do not form a matching of $G$. This means that
Proposition~\ref{jun5-24}(b) does not extend to paths of length $3$.  

\begin{figure}[ht]
\begin{tikzpicture}[scale=2,thick]
		\tikzstyle{every node}=[minimum width=0pt, inner
		sep=1.8pt, circle]
			\draw (-0.24,1.47) node[draw] (0) { \tiny $t_3$};
			\draw (-0.67,0.75) node[draw] (1) { \tiny $t_1$};
			\draw (0.11,0.76) node[draw] (2) { \tiny $t_2$};
			\draw (-0.64,0.08) node[draw] (3) { \tiny
			$t_6$};
			\draw (0.17,0.1) node[draw] (4) { \tiny $t_7$};
			\draw (0.46,1.43) node[draw] (5) { \tiny $t_4$};
			\draw (1.03,1.23) node[draw] (6) { \tiny
			$t_5$};
			\draw  (0) edge (1);
			\draw  (0) edge (2);
			\draw  (1) edge (2);
			\draw  (1) edge (3);
			\draw  (2) edge (4);
			\draw  (3) edge (4);
			\draw  (0) edge (5);
			\draw  (5) edge (6);
		\end{tikzpicture}
\caption{Unmixed graph with $\mathcal{G}_J$ connected
and $\mathcal{G}_J^{(C_1,C_6)}$ not connected.}\label{figure-enrique1}
\end{figure}
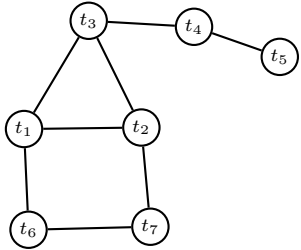

\end{example}

\begin{example}\label{example-conjetura3}
Let $G$ be the graph of Figure~\ref{figure-enrique1} and let $H$ be
the subgraph of $G$ induced by
the vertices $t_1, t_2, t_3, t_6, t_7$. Then, $H$ consists of two cycles:
$\{t_1,t_2,t_3,t_1\}$ and  $\{t_1,t_2,t_7,t_6,t_1\}$.  Using
\textit{Macaulay}$2$ \cite{mac2}, one obtains that $G$ is an unmixed
graph with $\alpha_0(G)=3$ whose minimal vertex covers are: 
$$
C_1=\{t_1,t_3,t_7\},\  C_2=\{t_1,t_2,t_7\},\  C_3=\{t_1,t_2,t_6\}
\mbox{ and } C_4=\{t_2,t_3,t_6\}. 
$$
\quad Letting $J=I_c(H)$, the graph $\mathcal{G}_J$ of $J$ is the path (with $4$
vertices)  $\mathcal{P}=\{C_1, C_2, C_3, C_4\}$. Furthermore,  
the subgraph $\mathcal{G}_J^{(C_i,C_j)}$
of $\mathcal{G}_J$ is connected, for all $1 \leq i < j \leq 4$.
Hence, by Proposition~\ref{Bigdeli-Herzog-j}, $J$ is linearly presented. 
But $H':=H[C_1 \triangle C_4]=\{t_1,t_2,t_7,t_6,t_1\}$ is a $4$-cycle, then 
$H'$ is an unmixed K\"onig graph and it is not Cohen--Macaulay.
Therefore, the converse of Theorem~\ref{noLinear-noCM} is not true.
\end{example}

\begin{appendix}

\section{Procedures}\label{Appendix}

\begin{procedure}\label{procedure1}
Computing the graph $\mathcal{G}_{I_c(G)}$ of the ideal of covers 
$I_c(G)$ of a graph $G$ using \textit{Macaulay}$2$ 
\cite{mac2}. This procedure corresponds to
Example~\ref{example-conjetura}. 
One can compute other examples by changing the
polynomial ring $S$ and the generators of the ideal $I$.
\begin{verbatim}
restart
load "EdgeIdeals.m2"
--The ti's correspond to the vertices of the graph G 
--and the yi's correspond to the vertex covers of G
--These should be changed to run other examples
S=QQ[t1,t2,t3,t4,t5,t6,t7,t8,t9,t10,y1,y2,y3,y4,y5,y6,y7,y8,y9,y10]
--This is the edge ideal of G
I=monomialIdeal(t1*t2,t2*t3,t3*t4,t1*t4,t1*t3,t2*t4,t1*t5,t5*t6,t6*t2,
t1*t8,t2*t9,t3*t7,t4*t10,t7*t8,t10*t9,t8*t9,t3*t9,t4*t8)
--This is the ideal of covers of G
Ic=dual I
--This checks whether or not G is unmixed
degrees Ic
M=coker gens gb Ic
pdim M 
--The syzygy matrix of Ic(G) is:
syz gens Ic
--G has 8 minimal vertex covers corresponding to y1,...,y8
f=matrix{{y1},{y2},{y3},{y4},{y5},{y6},{y7},{y8}}
J=ideal(transpose(f)*(syz gens Ic))
--Eliminate the generators of J of degree > 2 to obtain: 
J1=ideal(apply(flatten entries gens J,x-> if degree(x)=={2}
 then x else 0))
--number of independent linear syzygies of Ic(G):
numgensJ1=#(set(apply(flatten entries gens J,x-> 
if degree(x)=={2} then x else 0))-{0})
fsub=gens sub(ideal(toList (set(apply(flatten entries gens J,
x-> if degree(x)=={2} then x else 0))-{0})),
{t1=>1,t2=>1,t3=>1,t4=>1,t5=>1,t6=>1,t7=>1,t8=>1,t9=>1,t10=>1})
--This gives the graph of Ic(G) if the columns of the 
--linear syzygy matrix of Ic(G) are linearly independent: 
fsub1=monomialIdeal gens gb ideal apply(#(flatten entries fsub),
n->someTerms((flatten entries fsub)#n,1,1)*
someTerms((flatten entries fsub)#n,0,1))
GIc=graph(fsub1)
isBipartite GIc
allOddHoles(GIc)
isConnected(GIc)
L=ass Ic
p=(n)->gens gb ideal(flatten mingens(quotient(Ic,L#n)/Ic))
minA=monomialIdeal(apply(0..#L-1,p))
--This computes the v-number of Ic(G)
vnumber=min flatten degrees minA
--The regularity of Ic(G)
regularity M
--minimal free resolution of Ic(G)
Mres=res M
--This is the syzygy matrix of Ic(G)
Mres.dd_2
--Linearly independent linear syzygies of Ic(G)
submatrix(Mres.dd_2,{0..5})
--rank of the linear syzygy matrix
--this rank is number of minimal vertex covers -1 
--if Ic(G) is linearly presented
rank submatrix(Mres.dd_2,{0..5})
--Minors of the linear syzygy matrix of Ic(G) if 
--columns are linearly independent 
minors(6,submatrix(Mres.dd_2,{0..5}))
--codim >=2 if Ic(G) is linearly presented:
codim minors(6,submatrix(Mres.dd_2,{0..5}))
--Now we use the edge ideal I to compute the graph of Ic(G) 
--using a method that works in general 
G=flatten apply(ass(I),x-> entries gens gb (x))
A1=toList subsets(G,2)
B1=toList set apply(A1,x->toList x)
--This gives the edges of the graph of Ic(G) 
graphIc=apply(0..#B1-1,n-> if #(set((B1#n)#0)*set((B1#n)#1))==
codim(I)-1 then  (set((B1#n)#0),set((B1#n)#1)) else set{})
graph1=toList set graphIc
--This is the number of edges of the graph of Ic(G)
#(set graph1)-1
--If the following two numbers are equal 
--the linear syzygy matrix of Ic(G) has linearly independent columns 
#(set graph1)-1==numgensJ1
\end{verbatim}
\end{procedure}
\end{appendix}

\section*{Acknowledgments.} 
We thank the reviewers for a careful
reading of the paper and for the improvements suggested. 
\textit{Macaulay}$2$ \cite{mac2} was used to implement 
algorithms to compute the graph of ideals of covers and the v-number
of monomial ideals. 

\section*{Statements and Declarations}  
On behalf of all authors, the corresponding author states that there is no conflict of interest.

No funding was received for conducting this study.

The authors have no relevant financial or non-financial interests to disclose.

Data sharing is not applicable to this article as no datasets were
generated or analyzed during the current study.

\bibliographystyle{plain}

\begin{thebibliography}{10}
\bibitem{AL}{W. W. Adams and P. Loustaunau, {\em An Introduction to 
Gr\"obner Bases\/}, Graduate Studies in Mathematics {\bf 3}, American Mathematical 
Society, Providence, RI, 1994.} 

\bibitem{aigner} M. Aigner, {\it Combinatorial Theory\/}, Springer,
1997. 

\bibitem{Bigdeli-etal} M. Bigdeli, J. Herzog and R. Zaare-Nahandi, 
On the index of powers of edge ideals, Comm. Algebra {\bf 46} (2018),
no. 3, 1080--1095.

\bibitem{Biswas-Mandal} P. Biswas and M. Mandal, A study of v-number for some monomial ideals. Collect.
Math. (2024). \url{https://doi.org/10.1007/s13348-024-00451-x}.

\bibitem{BE} 
D. A. Buchsbaum and  D. Eisenbud, 
What makes a complex exact?, J. Algebra {\bf 25} (1973), 259--268.

\bibitem{Ivan-Cruz-Reyes} I. D. Castrill\'on, R. Cruz and E. Reyes,
On well-covered, vertex decomposable and Cohen--Macaulay graphs,
Electron. J. Comb. {\bf 23} (2) (2016), Paper No. 39, 17 pp..

\bibitem{civan} Y. Civan, The {\rm v}-number and Castelnuovo-Mumford regularity of
graphs, J. Algebraic Combin. {\bf 57} (2023), 161--169.

\bibitem{min-dis-generalized} S. M. Cooper, A. Seceleanu, S. O. Toh\v{a}neanu,
M. Vaz Pinto and R. H. Villarreal, 
Generalized minimum distance functions and algebraic invariants of
Geramita ideals, Adv. in Appl. Math. {\bf 112} (2020), 101940.

\bibitem{Crupi-Rinaldo-Terai} M. Crupi, G. Rinaldo and N. Terai, Cohen--Macaulay edge ideal whose
height is half of the number of vertices, Nagoya
Math. J. {\bf 201} (2011), 117--131.

\bibitem{Dao-Eisenbud} H. Dao and D. Eisenbud, 
Linearity of free resolutions of monomial ideals, Res. Math. Sci. 
{\bf 9} (2022), no. 2, Paper No. 35, 15 pp.

\bibitem{ER}{J.~A. Eagon and V. Reiner, Resolutions of 
Stanley-Reisner rings and Alexander duality, J. 
Pure Appl. Algebra {\bf 130} (1998), 265--275.}

\bibitem{Eisen}{D. Eisenbud, {\it Commutative Algebra with a view
toward Algebraic Geometry\/}, Graduate
Texts in  Mathematics {\bf 150}, Springer-Verlag, New York, 1995.}

 \bibitem{favaron}  O. Favaron, Very well-covered graphs, Discrete Math. {\bf 42} (1982), 
177--187.

\bibitem{Ficarra} A. Ficarra, Simon conjecture and the
$\text{v}$-number of monomial ideals, Collect. Math. (2024).\\  
\url{https://doi.org/10.1007/s13348-024-00441-z}.

\bibitem{alexdual} I. Gitler, E. Reyes and R. H. Villarreal, Blowup 
algebras of ideals of vertex covers of bipartite graphs, Contemp. 
Math. {\bf 376} (2005), 273--279. 

\bibitem{godsil}{C. Godsil and G. Royle, {\it Algebraic Graph Theory\/},
Graduate Texts in  Mathematics {\bf 207}, Springer, New York, 2001.}

\bibitem{mac2} D. Grayson and M. Stillman,
{\em Macaulay\/}$2$, {\em a software system for research in algebraic geometry},
available at \text{https://macaulay2.com/}.

\bibitem{im-vnumber} G. Grisalde, E. Reyes and R. H. Villarreal, 
Induced matchings and the v-number of graded ideals, Mathematics {\bf 9} (2021), no. 22, 2860.

\bibitem{Har}{F. Harary, {\it Graph Theory\/}, Addison-Wesley, 
Reading, MA, 1972.}

\bibitem{v-number} D. Jaramillo and R. H. Villarreal, The
v-number of edge ideals, J. Combin. Theory Ser. A {\bf 177} (2021),
Paper 105310, 35 pp. 

\bibitem{Kimura-Terai} K. Kimura, M. R. Pournaki, S. A. Seyed
Fakhari, N. Terai and S. A. Yassemi, A glimpse to most of the old and
new results on very well-covered graphs from the viewpoint of
commutative algebra,  Res. Math. Sci. {\bf 9} (2022), no. 2, 
Paper No. 29, 18 pp. 

\bibitem{Kumar-Nanduri-Saha} M. Kumar, R. Nanduri and K. Saha, 
The slope of the v-function and the Waldschmidt constant, J. Pure
Appl. Algebra {\bf 229} (2025), no. 2, 107881. 

\bibitem{Lyu1}{G. Lyubeznik, The minimal non-Cohen--Macaulay 
monomial ideals, J. Pure Appl. Algebra {\bf 51} (1988),
 261--266.}

\bibitem{MSJ} E. Manouchehri and A. Soleyman Jahan, 
The linear syzygy graph of a monomial ideal and linear resolutions, 
Czechoslovak Math. J. {\bf 71} (146) (2021), no. 3, 785--802.

\bibitem{Mats}{H. Matsumura, {\it Commutative Ring Theory\/},
Cambridge
Studies in Advanced Mathematics {\bf 8},
Cambridge University Press, 1986.}

\bibitem{MRV} S. Morey, E. Reyes and R. H. Villarreal, 
Cohen--Macaulay, shellable and unmixed clutters with a perfect
matching of K\"{o}nig type, J. Pure Appl. 
Algebra {\bf 212} (2008), no. 7, 1770--1786.

\bibitem{edge-ideals} S. Morey and R. H. Villarreal, Edge ideals:
algebraic and combinatorial properties, in {\it Progress in
Commutative Algebra, Combinatorics and Homology, Vol. 1\/} (C.
Francisco, L. C.  Klingler, S. Sather-Wagstaff and J. C. Vassilev,
Eds.), De Gruyter, Berlin, 2012, pp. 85--126. 

\bibitem{unmixed-c-m}  Y. Pitones, E. Reyes and R. H. Villarreal, 
Unmixed and Cohen--Macaulay weighted oriented
K\"onig graphs,  
Studia Sci. Math. Hungar. {\bf 58} (2021), no. 3, 276--292.

\bibitem{Ramos-Simis} Z. Ramos and A. Simis, 
\textit{Determinantal ideals of square linear matrices}, 
Springer, Cham, 2024.

\bibitem{Saha} K. Saha, The v-number and Castelnuovo-Mumford
regularity of cover ideals of graphs, 
Int. Math. Res. Not. IMRN 2024, no. 11, 
9010--9019.   

\bibitem{Saha-Gorenstein} K. Saha and N. Kotal, 
On the v-number of Gorenstein ideals and Frobenius powers, 
Bull. Malays. Math. Sci. Soc.  {\bf 47} (2024), no. 6, Paper No. 167, 17 pp.

\bibitem{saha-sengupta} K. Saha and I. Sengupta, The {\rm v}-number
of monomial ideals, J. Algebraic Combin. {\bf 56} (2022), 903--927. 

\bibitem{birational}  A. Simis and R. H. Villarreal, 
Linear syzygies and birational combinatorics, Results Math. 
{\bf 48} (2005), no. 3-4, 326--343.

\bibitem{dsmith}{D.~E. Smith, On the Cohen--Macaulay property 
in commutative algebra and simplicial topology, Pacific J. Math. 
{\bf 141} (1990), 165--196.}

\bibitem{terai} N. Terai, Alexander duality theorem and
Stanley-Reisner rings, \textit{Free Resolutions of Coordinate Rings of
Projective Varieties and Related Topics} (Japanese) (Kyoto, 1998),
{\bf 1078} (1999), 174--184.

\bibitem{van-tuyl-survey} A. Van Tuyl, 
A beginner's guide to edge and cover ideals, \textit{Monomial ideals,
computations and applications}, 63--94,
Lecture Notes in Math. {\bf 2083}, Springer, Heidelberg, 2013.

\bibitem{cmg}{R. H. Villarreal, Cohen--{M}acaulay graphs, Manuscripta
Math. {\bf 66} (1990), 277--293.}

\bibitem{unmixed} R. H. Villarreal, Unmixed bipartite graphs, Rev. 
Colombiana Mat. {\bf 41} (2007), no. 2, 393--395. 

\bibitem{monalg-rev} R. H. Villarreal, \textit{Monomial Algebras, Second Edition},
Monographs and Research Notes in Mathematics, Chapman and Hall/CRC,
Boca Raton, FL, 2015.

\bibitem{zaare-nahandi} R. Zaare-Nahandi, Cohen--Macaulayness of
bipartite graphs, revisited, Bull. Malays. Math. Sci. Soc. {\bf 38}
(2015), 1601--1607. 

\end{thebibliography}

\end{document}